\documentclass[11pt]{article}

\usepackage{graphicx}
\usepackage{multirow}
\usepackage{amsmath,amssymb,amsfonts}
\usepackage{amsthm}
\usepackage{mathrsfs}
\usepackage{xcolor}
\usepackage{textcomp}
\usepackage{manyfoot}
\usepackage{booktabs}
\usepackage{listings}
\usepackage{xspace}
\usepackage[margin=1in]{geometry}
\usepackage[numbers,sort&compress]{natbib}
\usepackage[hidelinks]{hyperref}
\usepackage{orcidlink}
\usepackage{enumitem}
\setlist[itemize,1]{label=\textendash}

\theoremstyle{plain}
\newtheorem{theorem}{Theorem}
\newtheorem{proposition}[theorem]{Proposition}
\newtheorem{lemma}[theorem]{Lemma}
\newtheorem{corollary}[theorem]{Corollary}

\theoremstyle{definition}
\newtheorem{example}{Example}
\newtheorem{remark}{Remark}

\theoremstyle{definition}
\newtheorem{definition}{Definition}

\newcommand{\X}{\mathcal X}
\newcommand{\BT}{\mathcal T}
\newcommand{\AOD}{\operatorname{AOD}}
\newcommand{\eps}{\varepsilon}
\newcommand{\op}{\odot}
\newcommand{\QoI}{\mathcal Q}
\newcommand{\R}{\mathbb R}
\newcommand{\Curv}{\mathcal C}
\newcommand{\diam}{\operatorname{diam}}

\newcommand{\wv}{\mathbf w}
\newcommand{\Fw}[1]{F_{#1,\wv}^{\eps}}
\newcommand{\AODfw}{\AOD_{f,\wv}}
\newcommand{\AODfwQ}{\AOD_{f,\wv,\QoI}}

\newcommand*\ie{\text{i.e.,\,}{}}

\begin{document}

\title{Local Second-Order Bounds for Aggregation-Order Variation in Density Fusion}

\author{Ratan Bahadur Thapa \orcidlink{0009-0000-2368-5928}\\
Institute for Artificial Intelligence, University of Stuttgart,\\
Germany}
\date{}

\maketitle

\begin{abstract}
Distributed statistical analyses often aggregate local posterior or predictive densities by repeated pairwise fusion. When the binary fusion rule is nonassociative, changing the ordered aggregation tree can change the final density and the reported posterior summaries. We study this aggregation-order variation for smooth $f$-divergence balancing in a local density chart around a common reference density. Under square-root-transformed supplied weights and first-order nondegeneracy at each generated merge, every tree in a fixed finite family shares the same first-order coefficient, whereas the first potentially tree-dependent term appears at second order. We derive the explicit second-order coefficient, propagate it through every tree in the family by a recursion for the tree-indexed second-order density coefficient, and show that a finite three-state posterior contrast detects the resulting discrepancy. The same coefficient determines density-level and quantity-of-interest diameters, rotation bounds, lower bounds, a limitation of scalar calibration, and a local corrected chart representation. The resulting expansions are uniform over each fixed finite tree family satisfying the stated local conditions.
\end{abstract}



\section{Introduction}\label{sec:intro}
Distributed statistical analysis often computes local posterior or predictive density summaries and then combines them into a single distribution. Subset-posterior and divide-and-conquer Bayesian methods combine local posterior summaries~\cite{neiswanger2014asymptotically,scott2016bayes,srivastava2018scalable}. Predictive density methods and expert opinion pooling use related aggregation principles for probabilistic forecasts and expert distributions~\cite{stone1961opinion,genest1986combining,jacobs1995methods,gneiting2007strictly}. In many applications, privacy restrictions, memory limits, communication constraints, network topology, or hierarchical protocols require repeated pairwise fusion along an ordered binary aggregation tree~\cite{agarwal2013mergeable}.

Pairwise aggregation introduces an order question whenever the binary rule is nonassociative. Consider three local posterior densities $p_1,p_2,p_3$ and a binary fusion rule $\op$. The two outputs $(p_1\op p_2)\op p_3$ and $p_1\op(p_2\op p_3)$ may differ although both calculations use the same ordered inputs and the same rule. A bounded posterior mean, a tail probability, or a contrast between two hypotheses may then depend on the parenthesization. We call the resulting deterministic discrepancy aggregation-order variation.

We consider dominated probability densities, \ie densities defined with respect to a common dominating measure, and measure aggregation-order variation by a diameter over ordered binary aggregation trees. For fixed ordered inputs, the density-level diameter compares all outputs obtained by changing only the tree. To connect density variation with reported statistical summaries, we also use a quantity-of-interest diameter obtained by pairing each output with a bounded statistical functional, such as a bounded posterior mean, a posterior tail probability, a bounded loss, or a finite-state contrast.

The main analytic case is smooth $f$-divergence balancing. Given weighted endpoint densities $(a,r)$ and $(b,s)$, the binary rule chooses a density on the segment between $r$ and $s$ by balancing the two weighted endpoint divergences. Divergence-based statistical inference is classical, starting with Kullback--Leibler information, Ali--Silvey divergences, and Csisz\'ar divergences~\cite{kullback1951information,ali1966general,csiszar1967information}, and remains central in robust inference and density-based estimation~\cite{singh2021robust}. 

Let $\eps$ be a scalar local parameter. For each input density, let $h_i$ be the first-order zero-mass coefficient and let $B_i$ be the second-order zero-mass coefficient. Around a strictly positive reference density $p$, the local model is
\[p_i^\eps=p+\eps h_i+\eps^2B_i+o(\eps^2).\]
The common purely first-order convention is $B_i=0$. The binary balancing equation has a distinguished local root. Its first-order coefficient depends on endpoint weights through square roots. After replacing supplied weights $w_i$ by additive transformed leaf masses $G_i=\sqrt{w_i}$, every aggregation tree satisfying first-order nondegeneracy at its generated merges has the same first-order term. For a subtree $A$, the protocol carries $G_A=\sum_{i\in A}\sqrt{w_i}$; a merge of subtrees $A$ and $B$ uses the divergence-balancing weights $G_A^2$ and $G_B^2$. The first potentially tree-dependent coefficient occurs at second order.

\paragraph{\textbf{Our contributions are as follows.}} We formulate aggregation-order variation for pairwise density fusion as density-level and quantity-of-interest diameters over ordered binary aggregation trees. For smooth $f$-divergence balancing, we derive the local second-order coefficient that governs the tree-dependent term after square-root transformation of the supplied weights. We lift the binary coefficient to ordered binary trees satisfying first-order nondegeneracy at every generated merge by a recursion for the tree-indexed second-order density coefficient. We show, through a finite three-state posterior calculation, that a standard posterior contrast detects a nonzero order-$\eps^2$ discrepancy. We then derive rotation bounds, lower bounds, a limitation of scalar calibration, and a local corrected chart representation from the same second-order expansion. The corrected chart is local and asymptotic; finite-perturbation claims require separate remainder estimates.

\paragraph{\textbf{The paper is organized as follows.}} Section~\ref{sec:related} discusses related work, and Section~\ref{sec:aggregation} introduces binary aggregation trees and aggregation-order variation. Section~\ref{sec:local} proves the binary divergence expansion. Section~\ref{sec:trees} lifts the expansion to tree families satisfying the local hypotheses. Section~\ref{sec:consequences} gives the finite posterior witness and quantity-of-interest consequences. Section~\ref{sec:normalform} gives the corrected local chart. Section~\ref{sec:bounds} gives the diameter, rotation, lower-bound, and calibration statements, and Section~\ref{sec:conclusion} concludes. Appendices A--E give the extended proofs of the main results and the auxiliary derivations.

\section{Related Work}\label{sec:related}
Kullback--Leibler information, Ali--Silvey divergences, and Csisz\'ar divergences provide the classical calculus for divergence-based comparison of probability distributions \cite{kullback1951information,ali1966general,csiszar1967information}. Bregman centroids, information-geometric coordinates, and proper scoring rules give related convex principles for statistical summaries and predictive assessment \cite{banerjee2005clustering,nielsen2009sided,amari2016information,gneiting2007strictly}, including robust inference based on exponential-polynomial divergence \cite{singh2021robust}. We use the local Taylor expansion of $f$-divergences for a different purpose: to identify the coefficient through which a pairwise balancing rule transmits aggregation-tree dependence.

Distributed Bayesian computation combines subset posterior summaries under data partitioning through consensus Monte Carlo and embarrassingly parallel MCMC \cite{neiswanger2014asymptotically,scott2016bayes}. Wasserstein barycenter posterior aggregation gives a simultaneous geometric aggregate for subset posteriors \cite{agueh2011barycenters,srivastava2018scalable}. Density-based Bayesian modeling also appears in Bayesian quantile and expectile regression with discrete responses \cite{liu2021discrete}. These methods define or approximate a global posterior summary. We instead hold the mass-carrying divergence-balancing protocol and supplied weights fixed and quantify the local variation caused only by the aggregation tree.

Linear and logarithmic pooling and axiomatic analyses of probability aggregation study rules for combining expert distributions \cite{stone1961opinion,genest1986combining,jacobs1995methods}. Associativity determines whether repeated pooling is independent of aggregation order. We study the complementary regime in which a mass-carrying divergence-balancing protocol has a nonzero second-order reassociation coefficient detectable by ordinary statistical functionals.

Mergeable summaries give associative merge operations in database and streaming settings \cite{agarwal2013mergeable}. Thapa et al. prove exact order-invariance characterizations for local segment-valued density-fusion rules and also identify square-root effective weights from the local quadratic expansion of smooth endpoint-to-candidate $f$-divergence balancing \cite{thapa2026compositional}. We instead quantify the leading second-order aggregation-order variation when exact invariance fails.

\section{Density Aggregation}\label{sec:aggregation}
We now define the aggregation-order variation diameter over ordered binary trees and bound it by local reassociation defects.

Let $(\X,d)$ be a metric space, where elements of $\X$ represent uncertainty summaries and $d$ is the metric used to compare them. Let $\op:\X\times\X\to\X$ be a binary aggregation rule.  For an integer $n\ge2$, let $\BT_n$ denote the set of full ordered binary trees with $n$ leaves. For $T\in\BT_n$ and $\mathbf{x}=(x_1,\ldots,x_n)\in\X^n$, we write $F_T(\mathbf{x})$ for the output obtained by placing $x_i$ at leaf $i$ and applying $\op$ at every internal node. For a nonempty set $S\subseteq\X$, we write $\diam_{d}S=\sup\{d(u,v) \mid u,v\in S\}$ for the diameter of $S$ with respect to $d$.

\begin{definition}\label{def:puq}
For $n\ge2$ and $\mathbf{x}\in\X^n$, the aggregation-order variation diameter is
\[\AOD_{\op}(\mathbf{x})=\diam_{d}\{F_T(\mathbf{x}) \mid T\in\BT_n\}
 =\sup_{T,T'\in\BT_n}d(F_T(\mathbf{x}),F_{T'}(\mathbf{x})).\]
\end{definition}

Aggregation-order variation is the formal diameter of the outputs obtained from all ordered binary aggregation trees with the same ordered inputs and the same binary rule. 

When $\X$ is a class of probability densities on $(\Omega,\mathcal F,\mu)$, a quantity-of-interest class $\QoI$ is a set of measurable functions for which the relevant integrals are finite. The corresponding functional diameter is
\[\AOD_{\op,\QoI}(\mathbf{x})=
\sup_{\varphi\in\QoI}\,\, \sup_{T,T'\in\BT_n}
\left|\int_\Omega \varphi\{F_T(\mathbf{x})-F_{T'}(\mathbf{x})\}\,d\mu\right|.\]

If $\QoI$ is contained in the unit ball of $L^\infty(\mu)$, then $\AOD_{\op,\QoI}(\mathbf{x})$ is bounded by the $L^1$ diameter of the tree outputs.

Throughout, for an integrable signed density $g$, we use
\[
\|g\|_{\mathrm{TV}}=\int_\Omega |g|\,d\mu=\|g\|_1
\]
for the total variation norm of the signed measure $g\,d\mu$. For probability densities $q$ and $q'$, we use the total variation distance $d_{\mathrm{TV}}(q,q')=\frac12\|q-q'\|_1$. Thus,
\[
\sup_{\|\varphi\|_\infty\le1}\left|\int_\Omega\varphi g\,d\mu\right|=\|g\|_{\mathrm{TV}},
\]
whereas a leading coefficient stated for $d_{\mathrm{TV}}$ is one half of the corresponding $L^1$ coefficient.

For $x,y,z\in\X$, the associativity defect of the triple $(x,y,z)$ is
\[A(x,y,z)=d((x\op y)\op z,x\op(y\op z)).\]
A context is a tree expression with one hole. If $C[\cdot]$ is a context and the following supremum is finite, let
\[
\Lambda(C)=\sup\left\{\frac{d(C[u],C[v])}{d(u,v)} \mid u,v\in\X,\ d(u,v)>0\right\}.
\]
A rotation path from $T$ to $T'$ is a finite sequence of elementary reassociations between ordered binary trees; rotation connectivity is the standard combinatorial property of such trees~\cite{stasheff1963homotopy,mac1971categories}.

For the next proposition, let $T,T'\in\BT_n$ and let
$\gamma=(T_0,T_1,\ldots,T_m)$ be a rotation path from $T_0=T$ to $T_m=T'$. For $j=1,\ldots,m$, the $j$th rotation replaces a local subtree with values $(a_j,b_j,c_j)$ inside a surrounding context $C_j[\cdot]$. Assume that all constants $\Lambda(C_j)$ are finite. For a rotation $\rho$ on the path, we write $A_\rho$ for its local associativity defect and $\Lambda_\rho$ for the Lipschitz modulus of its surrounding context.

\begin{proposition}\label{prop:assoc}
For every rotation path $\gamma=(T_0,T_1,\ldots,T_m)$ from $T$ to $T'$,
\[
d(F_T(\mathbf{x}),F_{T'}(\mathbf{x}))
\le
\sum_{j=1}^{m}\Lambda(C_j)A(a_j,b_j,c_j).
\]
The aggregation-order variation satisfies
\[
\AOD_{\op}(\mathbf{x})\le
\sup_{T,T'\in\BT_n}\inf_{\gamma:T\leadsto T'}
\sum_{\rho\in\gamma}\Lambda_\rho A_\rho .
\]
\end{proposition}

\begin{proof}[Proof.]
For one elementary rotation,
\[F_{T_j}(\mathbf{x})=C_j[((a_j\op b_j)\op c_j)]\quad\text{and}\quad
F_{T_{j+1}}(\mathbf{x})=C_j[a_j\op(b_j\op c_j)].
\]
The definition of $\Lambda(C_j)$ gives
\[
d(F_{T_j}(\mathbf{x}),F_{T_{j+1}}(\mathbf{x}))\le \Lambda(C_j)A(a_j,b_j,c_j).
\]
Summing along $\gamma$ by the triangle inequality proves the first inequality. Taking the infimum over all paths and then the supremum over endpoint trees gives the diameter bound.
\end{proof}

\begin{corollary}\label{cor:assoc}
For a binary rule $\op$ on $\X$, the following statements are equivalent:
\begin{enumerate}
\item For all $n\ge2$, all $\mathbf{x}\in\X^n$, and all $T,T'\in\BT_n$, we have
$F_T(\mathbf{x})=F_{T'}(\mathbf{x})$.
\item For all $x,y,z\in\X$, we have $A(x,y,z)=0$.
\item The binary rule $\op$ is associative.
\end{enumerate}
\end{corollary}

\begin{proof}
By definition of $A$, statements~(2) and~(3) are equivalent. Under~(3),
repeated reassociation shows that all parenthesizations of
$x_1\op\cdots\op x_n$ have the same value, proving~(1). Conversely,
applying~(1) with $n=3$ to the two binary trees yields
$(x\op y)\op z=x\op(y\op z)$ for all $x,y,z\in\X$, proving~(3).
\end{proof}

\section{Local Balancing}\label{sec:local}

The preceding section isolates reassociation as the algebraic source of aggregation-order variation. We next compute the local defect for smooth $f$-divergence balancing. 

Let $(\Omega,\mathcal F,\mu)$ be a measure space, and let $p$ be a strictly positive probability density. For any measurable function $q$ with $q/p\in L^\infty(\mu)$,  we set $\|q\|_{p,\infty}=\|q/p\|_\infty$. 
Throughout, as $\eps\to0$, we write $O(\eps^k)$ when the remainder divided by $|\eps|^k$ remains bounded in the relevant norm, and $o(\eps^k)$ when this quotient converges to zero.
The local chart consists of densities
\[
r^\eps=p+\eps h_A+\eps^2B_A+o(\eps^2)\quad\text{and}\quad
s^\eps=p+\eps h_B+\eps^2B_B+o(\eps^2),
\]
where the remainders are measured in $\|\cdot\|_{p,\infty}$, and
\[
\int h_A\,d\mu=\int h_B\,d\mu=\int B_A\,d\mu=\int B_B\,d\mu=0.
\]
We assume that $h_A/p$, $h_B/p$, $B_A/p$, and $B_B/p$ are bounded. We write
\[
u_A=h_A/p,\quad u_B=h_B/p,\quad \Delta=u_B-u_A,
\quad
J=\int_\Omega \Delta^2p\,d\mu
\quad\text{and}\quad
L_3=\int_\Omega \Delta^3p\,d\mu.
\]

Let $f:(0,\infty)\to\R$ be convex and $C^3$ in a neighborhood of $1$. Throughout the local results,
\[
f(1)=0,
\quad f'(1)=0
\quad\text{and}\quad f''(1)>0.
\]
The associated $f$-divergence is
\[
D_f(r\|q)=\int_\Omega q f(r/q)\,d\mu.
\]
For endpoint weights $a,b>0$, the balancing rule selects a local solution $t_\eps\in[0,1]$ of
\[
aD_f(r^\eps\|(1-t)r^\eps+ts^\eps)=bD_f(s^\eps\|(1-t)r^\eps+ts^\eps).
\]
The local root is the branch converging to
$t_0=\frac{\sqrt b}{\sqrt a+\sqrt b}$.
We also write
$A_2=\frac{f''(1)}2$ and $A_3=\frac{f'''(1)}6$.
The nondegenerate case is $J>0$. If $J=0$, the two first-order endpoints coincide in $L^2(p\,d\mu)$ and root selection requires higher-order information.

\begin{lemma}\label{lem:chart}
For sufficiently small $|\eps|$, all segment densities $(1-t)r^\eps+ts^\eps$, $0\le t\le1$, are positive and uniformly comparable with $p$. In addition, $r^\eps/((1-t)r^\eps+ts^\eps)$ and $s^\eps/((1-t)r^\eps+ts^\eps)$ converge uniformly to $1$.
\end{lemma}

\begin{proof}[Proof.]
The bounded-ratio assumptions give $r^\eps/p=1+O(\eps)$ and $s^\eps/p=1+O(\eps)$ uniformly. Convexity of the segment gives $((1-t)r^\eps+ts^\eps)/p=1+O(\eps)$ uniformly in $t$. Positivity and two-sided comparability follow for sufficiently small $|\eps|$. Dividing the endpoint ratios by the segment ratio gives uniform convergence of both likelihood ratios to $1$.

\end{proof}

The next lemma gives the root selection used by the binary expansion. The notation is that of Section~\ref{sec:local}, and the nondegeneracy condition $J>0$ holds.

\begin{lemma}\label{lem:root}
There exist $\eta>0$ and $\eps_0>0$ such that, for $0<|\eps|<\eps_0$, the balancing equation has exactly one solution in $(t_0-\eta,t_0+\eta)$. The solution converges to $t_0$ as $\eps\to0$.
\end{lemma}

\begin{proof}[Proof.]
Let
\[
\Phi_\eps(t)=aD_f(r^\eps\|(1-t)r^\eps+ts^\eps)-bD_f(s^\eps\|(1-t)r^\eps+ts^\eps).
\]
The local Taylor expansion of the two divergences gives the uniform expansion, for $t$ in a neighborhood of $t_0$,
\[
\eps^{-2}\Phi_\eps(t)= \Phi_2(t)+\eps\Phi_3(t)+o(\eps),
\]
where $\Phi_3$ is bounded uniformly on that neighborhood and
\[
\Phi_2(t)=A_2J\{at^2-b(1-t)^2\}.
\]
The corresponding derivative expansion is
$\eps^{-2}\partial_t\Phi_\eps(t)= \Phi_2'(t)+o(1)$
uniformly near $t_0$. The bounded-ratio assumptions and the $C^3$ expansion of $f$ near $1$ justify the uniform expansion of both $\Phi_\eps$ and $\partial_t\Phi_\eps$ on that neighborhood. Indeed, if $p_t^\eps=(1-t)r^\eps+ts^\eps$, then differentiating the exact integrand gives
\[
\partial_t\left[p_t^\eps f\left(\frac{r^\eps}{p_t^\eps}\right)\right]
=(s^\eps-r^\eps)\left\{f\left(\frac{r^\eps}{p_t^\eps}\right)
-\frac{r^\eps}{p_t^\eps}f'\left(\frac{r^\eps}{p_t^\eps}\right)\right\},
\]
and the same identity with $s^\eps$ in place of $r^\eps$. Lemma~\ref{lem:chart}, bounded density ratios, and the $C^3$ expansion of $f$ near $1$ give the stated derivative expansion uniformly in $t$.
More explicitly, as $z\to0$,
\[
f'(1+z)=f''(1)z+\frac12f'''(1)z^2+o(z^2).
\]
Lemma~\ref{lem:chart} makes the endpoint-to-segment ratio increments converge uniformly to zero on the chosen compact neighborhood of $t_0$. A uniform Taylor bound and domination by an integrable multiple of $p$ then give the stated uniform derivative remainder.
The root $t_0$ is simple because
\[
\Phi_2'(t_0)=2A_2J\{at_0+b(1-t_0)\}>0.
\]
Choose $\eta>0$ with $\eta<\min\{t_0,1-t_0\}$ so that $\Phi_2'(t)>c>0$ for $|t-t_0|<\eta$. The stated derivative expansion gives $\partial_t\Phi_\eps(t)>0$ on the same interval for sufficiently small nonzero $\eps$. Therefore, $\Phi_\eps$ is strictly increasing on that interval. Since $\Phi_2(t_0-\eta)<0<\Phi_2(t_0+\eta)$ after shrinking $\eta$ if needed, the uniform expansion gives opposite signs for $\Phi_\eps$ at the two endpoints. The intermediate value theorem gives existence, and strict monotonicity gives uniqueness. Applying the endpoint sign argument on every smaller neighborhood of $t_0$ proves that $t_\eps\to t_0$.
\end{proof}

The next theorem is stated under the nondegeneracy condition $J>0$ and with the local root from Lemma~\ref{lem:root}.

\begin{theorem}\label{thm:binary}
For endpoint weights $a,b>0$, the local balancing root satisfies
\[
t_\eps=t_0+\eps\tau(a,b;u_A,u_B)+o(\eps),
\]
where
\[
\tau(a,b;u_A,u_B)=
\frac{A_3L_3\{a t_0^3+b(1-t_0)^3\}}
{2A_2J\{a t_0+b(1-t_0)\}}.
\]
The fused segment density has expansion
\[
(1-t_\eps)r^\eps+t_\eps s^\eps
=p+\eps H_{AB}+\eps^2B_{AB}+o(\eps^2),
\]
with
$H_{AB}=(1-t_0)h_A+t_0h_B$
and
$B_{AB}=(1-t_0)B_A+t_0B_B+\tau(a,b;u_A,u_B)(h_B-h_A)$.
\end{theorem}

The condition $J>0$ makes the denominator in $\tau$ positive. The coefficient $\tau$ depends only on the first-order perturbation ratios $u_A,u_B$, the endpoint weights, and the generator through the ratio $f'''(1)/f''(1)$.

\begin{proof}[Proof.]
Let $p_t^\eps=(1-t)r^\eps+ts^\eps$ and, within the proof, write $B_A=pb_A$ and $B_B=pb_B$. Lemma~\ref{lem:chart} permits uniform Taylor expansion in $t$. Direct division gives
\[
\frac{r^\eps}{p_t^\eps}=1-\eps t\Delta+
\eps^2t\{\Delta(u_A+t\Delta)-(b_B-b_A)\}+o(\eps^2) \quad\text{and}
\]
\[
\frac{s^\eps}{p_t^\eps}=1+\eps(1-t)\Delta+
\eps^2(1-t)\{(b_B-b_A)-\Delta(u_A+t\Delta)\}+o(\eps^2).
\]
Since $f(1+z)=A_2z^2+A_3z^3+o(z^3)$, let
\[
M=\int \Delta(b_B-b_A)p\,d\mu \,\,\,\,\text{and}\,\,\,\,
K(t)=\int \Delta^2(u_A+t\Delta)p\,d\mu .
\]
Multiplication by $p_t^\eps=p\{1+\eps(u_A+t\Delta)+O(\eps^2)\}$ yields
\[
\begin{aligned}
D_f(r^\eps\|p_t^\eps)&=A_2\eps^2t^2J
+\eps^3\{A_2t^2(2M-K(t))-A_3t^3L_3\}+o(\eps^3),\\
D_f(s^\eps\|p_t^\eps)&=A_2\eps^2(1-t)^2J
+\eps^3\{A_2(1-t)^2(2M-K(t))+A_3(1-t)^3L_3\}+o(\eps^3).
\end{aligned}
\]
Let
\[
\Phi_\eps(t)=aD_f(r^\eps\|p_t^\eps)-bD_f(s^\eps\|p_t^\eps).
\]
Then,
\[
\Phi_\eps(t)=\eps^2\Phi_2(t)+\eps^3\Phi_3(t)+o(\eps^3) \,\,\,\,\text{and}\,\,\,\,
\Phi_2(t)=A_2J\{at^2-b(1-t)^2\},
\]
where $\Phi_3(t)$ denotes the coefficient of $\eps^3$ determined by the preceding two divergence expansions.
The leading equation $\Phi_2(t)=0$ has the root $t_0=\sqrt b/(\sqrt a+\sqrt b)$ in $[0,1]$, and
\[
\Phi_2'(t_0)=2A_2J\{at_0+b(1-t_0)\}>0.
\]
The expansion also gives $t_\eps-t_0=O(\eps)$. Indeed, dividing the root equation by $\eps^2$ gives
\[
0=\Phi_2(t_\eps)+\eps\Phi_3(t_\eps)+o(\eps),
\]
Since $\Phi_2'(t_0)>0$, there are $c_0>0$ and a neighborhood $I$ of $t_0$ such that
\[
c_0|t-t_0|\le |\Phi_2(t)|,\qquad t\in I.
\]
The convergence $t_\eps\to t_0$, local boundedness of $\Phi_3$, and the root equation therefore give
\[
c_0|t_\eps-t_0|
\le |\Phi_2(t_\eps)|
\le C|\eps|+o(|\eps|),
\]
and consequently $t_\eps-t_0=O(\eps)$. Taylor expansion at $t_0$ and continuity of the explicitly stated coefficient $\Phi_3$ give
\[
\Phi_2(t_\eps)=\Phi_2'(t_0)(t_\eps-t_0)+o(\eps)
\quad\text{and}\quad
\Phi_3(t_\eps)=\Phi_3(t_0)+o(1).
\]
Consequently, the root equation satisfies
\[
0=\Phi_2'(t_0)(t_\eps-t_0)+\eps\Phi_3(t_0)+o(\eps).
\]
Dividing by $\eps$ yields
\[
\frac{t_\eps-t_0}{\eps}=-\frac{\Phi_3(t_0)}{\Phi_2'(t_0)}+o(1).
\]
Thus, $t_\eps=t_0+\eps t_1+o(\eps)$ with $t_1=-\Phi_3(t_0)/\Phi_2'(t_0)$.
The terms containing $M$ and $K(t_0)$ have the common factor $at_0^2-b(1-t_0)^2$ and cancel. The remaining cubic term is
\[
\Phi_3(t_0)=-A_3L_3\{a t_0^3+b(1-t_0)^3\}.
\]
Solving for $t_1$ gives the stated expression for $\tau$. Finally,
\[
(1-t_\eps)r^\eps+t_\eps s^\eps
=p+\eps\{(1-t_0)h_A+t_0h_B\}
+\eps^2\{(1-t_0)B_A+t_0B_B+t_1(h_B-h_A)\}+o(\eps^2),
\]
which proves the density expansion.
\end{proof}

\section{Tree Expansions}\label{sec:trees}

We now lift the binary formula to ordered aggregation trees and state the additional local conditions needed for the recursion. After the square-root transformation of weights, the first-order balancing coefficient becomes an ordinary weighted average. The second-order terms remain tree-dependent through the placement of the local $\tau$ contributions.

For leaves $i=1,\ldots,n$, let $\wv=(w_1,\ldots,w_n)\in(0,\infty)^n$ be the supplied weight vector, let $G_i=\sqrt{w_i}$, and
\[
p_i^\eps=p+\eps h_i+\eps^2B_i+o(\eps^2),
\]
where the remainder is in $\|\cdot\|_{p,\infty}$, the coefficients have zero mass, and $h_i/p$ and $B_i/p$ are bounded. The common special case $p_i^\eps=p+\eps h_i+o(\eps^2)$ is obtained by setting $B_i=0$. For each leaf, we set
$H_i=h_i$. For a block $A$ of leaves, we write $G_A=\sum_{i\in A}G_i$. If a node merges blocks $A$ and $B$, set
\[
\theta_{AB}=\frac{G_B}{G_A+G_B}
\,\,\,\,\text{and}\,\,\,\,\,
H_{AB}=\frac{G_AH_A+G_BH_B}{G_A+G_B}.
\]
For the second-order recursion, start with the leaf coefficients $B_i$ and let
\[
B_{AB}=(1-\theta_{AB})B_A+\theta_{AB}B_B+\tau(G_A^2,G_B^2;H_A/p,H_B/p)(H_B-H_A),
\]
whenever the merge is nondegenerate. The notation $AB$ denotes the ordered union of the adjacent leaf blocks.

An admissible tree set for the local chart is a nonempty subset $\mathcal U\subseteq\BT_n$ such that all local summaries generated by the recursion over trees in $\mathcal U$ have bounded $p$-ratios and every generated merge is nondegenerate. For a merge of blocks $A$ and $B$, nondegeneracy means
\[
J(A,B)=\int_\Omega\left(\frac{H_B-H_A}{p}\right)^2p\,d\mu>0.
\]
For $T\in\mathcal U$, the mass-carrying divergence-balancing protocol associates a state $(G_A,q_A^\eps)$ with every generated block $A$. At leaf $i$, the state is $(G_i,p_i^\eps)$. At a merge of adjacent blocks $A$ and $B$, the protocol sets

\[
G_{AB}=G_A+G_B
\quad\text{and}\quad
q_{AB}^\eps=(1-t_{AB}^\eps)q_A^\eps+t_{AB}^\eps q_B^\eps,
\]
where $t_{AB}^\eps$ is the distinguished local solution from Lemma~\ref{lem:root} of
\[
G_A^2D_f(q_A^\eps\|q_{AB}^\eps)
=G_B^2D_f(q_B^\eps\|q_{AB}^\eps).
\]
The root density is denoted by $\Fw{T}$. Thus, a merge is determined by both density components and carried masses. Admissibility and Lemma~\ref{lem:root} determine the local branch at every generated merge for sufficiently small $|\eps|$. A merge with $J(A,B)=0$ requires a separate higher-order analysis and is excluded from $\mathcal U$.

The next theorem is stated for a fixed number of leaves and an admissible tree set $\mathcal U$.

\begin{theorem}\label{thm:tree}
Every ordered tree $T\in\mathcal U$ satisfies
\[
\Fw{T}=p+\eps H+\eps^2B_T+o(\eps^2),
\]
in $\|\cdot\|_{p,\infty}$, with the remainder uniform over $T\in\mathcal U$,
where
\[
H=\frac{\sum_{i=1}^nG_ih_i}{\sum_{i=1}^nG_i}
\]
is independent of $T$, and $B_T$ is obtained from the stated recursion.
\end{theorem}

\begin{proof}[Proof.]
The proof is by induction over the tree. At each leaf, $H_i=h_i$ and $B_i$ is the second-order coefficient of $p_i^\eps$. Suppose an internal node merges already expanded children $A$ and $B$. Theorem~\ref{thm:binary} applies with endpoint weights $G_A^2$ and $G_B^2$. Since
\[
\frac{\sqrt{G_B^2}}{\sqrt{G_A^2}+\sqrt{G_B^2}}=\frac{G_B}{G_A+G_B},
\]
at every induction step, admissibility supplies the bounded-ratio and $J(A,B)>0$ hypotheses for the generated child states.
The first-order term is $H_{AB}=(G_AH_A+G_BH_B)/(G_A+G_B)$. Iterating the formula gives the root coefficient $H=(\sum_iG_ih_i)/(\sum_iG_i)$, independent of the tree. The second-order term is exactly the stated recursion, because Theorem~\ref{thm:binary} adds the local term $\tau(G_A^2,G_B^2;H_A/p,H_B/p)(H_B-H_A)$ at the merge. For fixed $n$, the set $\mathcal U$ is finite because $\mathcal U\subseteq\BT_n$, and every tree has $n-1$ internal nodes. Lemma~\ref{lem:chart} gives one neighborhood of $p$ on which all generated endpoint and segment ratios remain bounded after shrinking the neighborhood if necessary. Because $\mathcal U$ is finite and each tree has finitely many internal nodes, the finitely many merge-specific $o(\eps^2)$ remainders admit a common bound after shrinking the local neighborhood. Structural induction therefore gives a root remainder that is $o(\eps^2)$ in $\|\cdot\|_{p,\infty}$ uniformly over $T\in\mathcal U$.
\end{proof}

If $f'''(1)=0$, then $\tau=0$ at every admissible merge, so the second-order recursion is weighted averaging and $B_T$ is independent of the aggregation tree. Consequently, for every $T,T'\in\mathcal U$,
\[
\Fw{T}-\Fw{T'}=o(\eps^2)
\]
in $\|\cdot\|_{p,\infty}$, uniformly over the admissible tree set. The corresponding restricted aggregation-order variation is therefore $o(\eps^2)$ for the metric induced by $\|\cdot\|_{p,\infty}$, and likewise for any output metric locally bounded by a constant multiple of this norm on the generated chart. Any nonzero tree dependence must then arise at higher order. More generally, the non-averaging second-order contribution at an individual admissible merge vanishes whenever either $f'''(1)=0$ or the corresponding cubic contrast moment $L_3$ is zero.

For a local rotation with three adjacent blocks $A,B,C$, let $B^{(AB)C}$ and $B^{A(BC)}$ denote the second-order coefficients produced by the two bracketings. We write
\[
\Curv(A,B,C)=B^{(AB)C}-B^{A(BC)}.
\]

The notation below uses the hypotheses and recursion of Theorem~\ref{thm:tree}. Assume that the two bracketings under comparison belong to an admissible tree set.

For the next proposition, let $H_0$ denote a common first-order coefficient and let $U,V$ denote zero-mass second-order coefficients in the generated local chart. Assume that the output metric satisfies
\[d(p+\eps H_0+\eps^2U+o(\eps^2),p+\eps H_0+\eps^2V+o(\eps^2))
=\eps^2\|U-V\|+o(\eps^2)
\]
uniformly over that chart. For $d(q,q')=\|q-q'\|_1$, the coefficient norm is $\|\cdot\|_1$. For the probability total-variation distance $d_{\mathrm{TV}}(q,q')=\frac12\|q-q'\|_1$, it is $\frac12\|\cdot\|_1$. The dual norm over the full unit ball of $L^\infty(\mu)$ equals $\|\cdot\|_1$; a restricted test-function class generally induces only a functional seminorm.

\begin{proposition}\label{prop:rotation}
\[
\Fw{(AB)C}-\Fw{A(BC)}=\eps^2\Curv(A,B,C)+o(\eps^2),
\]
and
\[
d(\Fw{(AB)C},\Fw{A(BC)})=\eps^2\|\Curv(A,B,C)\|+o(\eps^2).
\]
\end{proposition}

\begin{proof}[Proof.]
Theorem~\ref{thm:tree} gives the same zeroth-order density $p$ and the same first-order term
\[
\frac{G_AH_A+G_BH_B+G_CH_C}{G_A+G_B+G_C}
\]
for both bracketings. Subtracting the two expansions cancels those two common terms and leaves the second-order difference. The metric statement follows by applying the stated local metric expansion with $U=B^{(AB)C}$ and $V=B^{A(BC)}$.
\end{proof}

Section~\ref{sec:normalform} gives the corrected local chart. It carries the same transformed mass and first-order coefficient, while replacing the generated second-order coefficient by weighted averaging of zero-mass second-order density coefficients.

\section{Posterior Contrasts}\label{sec:consequences}

We next show that the second-order term has an observable statistical consequence. A finite three-point posterior-fusion calculation gives a signed second-order density coefficient, and pairing it with a posterior contrast gives a nonzero quantity-of-interest discrepancy.

Let $\Omega_0=\{1,2,3\}$ with counting measure and reference density $p=(1/3,1/3,1/3)$. Let $\wv=(1,1,1)$, and for a scalar $\alpha\ne0$ let
\[
h_1=\alpha(2,-1,-1),\quad
h_2=\alpha(-1,2,-1) \,\,\,\,\text{and}\,\,\,\,\,
h_3=\alpha(-1,-1,2).
\]
The scalar $\alpha$ is fixed while $\eps$ tends to zero. Positivity holds under the conservative sufficient bound $|\eps|<1/(6|\alpha|)$. Hence, the local densities $p_i^\eps=p+\eps h_i$ are positive and sum to one for sufficiently small $|\eps|$. The example uses the special leaf convention $B_i=0$. In the following theorem, $f$ satisfies the assumptions of Theorem~\ref{thm:binary} and $f'''(1)\ne0$.

\begin{theorem}\label{thm:witness}
For the three finite-state inputs above, square-root transformed-weight divergence balancing satisfies
\[
\Fw{((12)3)}-\Fw{1(23)}
=-\eps^2\frac{3\alpha^2}{4}\frac{f'''(1)}{f''(1)}(1,0,-1)+o(\eps^2).
\]
\end{theorem}

\begin{proof}[Proof.]
Let $A_2=f''(1)/2$, $A_3=f'''(1)/6$, and $k=A_3/A_2=f'''(1)/(3f''(1))$. At the lower merges,
\[\frac{h_2-h_1}{p}=9\alpha(-1,1,0)
\quad\text{and}\quad
\frac{h_3-h_2}{p}=9\alpha(0,-1,1),\]
and therefore
\[L_3(1,2)=\frac13\{(-9\alpha)^3+(9\alpha)^3+0^3\}=0\quad\text{and}\quad
L_3(2,3)=\frac13\{0^3+(-9\alpha)^3+(9\alpha)^3\}=0.\]
Thus, $\tau_{1,2}=\tau_{2,3}=0$, so both lower second-order contributions vanish.

For the upper merge in $((12)3)$, one has $G_{12}=2$, $G_3=1$, and $t_0=1/3$. The first-order lower summary is $H_{12}=(h_1+h_2)/2$. Let $\tau_{(12),3}=\tau(G_{12}^2,G_3^2;H_{12}/p,H_3/p)$. Substitution in Theorem~\ref{thm:binary} gives $\tau_{(12),3}=k\alpha/2$, and thus
\[B^{(12)3}=\tau_{(12),3}(H_3-H_{12})
=k\alpha^2\left(-\frac34,-\frac34,\frac32\right).
\]
For the upper merge in $1(23)$, let $\tau_{1,(23)}=\tau(G_1^2,G_{23}^2;H_1/p,H_{23}/p)$. The same calculation gives $\tau_{1,(23)}=-k\alpha/2$ and
\[
B^{1(23)}=k\alpha^2\left(\frac32,-\frac34,-\frac34\right).
\]
Therefore,
\[
B^{(12)3}-B^{1(23)}
=-\frac{9k\alpha^2}{4}(1,0,-1)
=-\frac{3\alpha^2}{4}\frac{f'''(1)}{f''(1)}(1,0,-1).
\]
Proposition~\ref{prop:rotation} gives the stated expansion.
\end{proof}

For the centered Kullback--Leibler generator $f(u)=u\log u-u+1$, $f''(1)=1$ and $f'''(1)=-1$. Theorem~\ref{thm:witness} gives
\[
\Fw{((12)3)}-\Fw{1(23)}
=\eps^2\frac{3\alpha^2}{4}(1,0,-1)+o(\eps^2).
\]
The posterior contrast $\varphi=(1,0,-1)$ satisfies
\[
\int \varphi \left\{ \Fw{((12)3)}-\Fw{1(23)}\right\}\,d\mu
=\eps^2\frac{3\alpha^2}{2}+o(\eps^2).
\]
Thus, the parenthesization changes a reported contrast at order $\eps^2$ while the three local posteriors, supplied weights, divergence generator, and mass-carrying protocol remain fixed.

Consider two aggregation trees with $\Fw{T}-\Fw{T'}=\eps^2\Xi+o(\eps^2)$
in $L^1(\mu)$, where $\int \Xi\,d\mu=0$. For the next proposition, $\varphi$ is a bounded test function satisfying $\int\varphi\Xi\,d\mu\ne0$.

\begin{proposition}\label{prop:qoi}
For every bounded test function $\varphi$ satisfying $\int\varphi\Xi\,d\mu\ne 0$,
\[
\left|\int\varphi(\Fw{T}-\Fw{T'})\,d\mu\right|
=\eps^2\left|\int\varphi\Xi\,d\mu\right|+o(\eps^2).
\]
For the unit ball of $L^\infty(\mu)$, the optimal leading coefficient is $\|\Xi\|_1$.
\end{proposition}

\begin{proof}[Proof.]
The first claim follows by pairing the $L^1$ expansion with the bounded function $\varphi$. For the unit ball of $L^\infty(\mu)$,
\[
\sup_{\|\varphi\|_\infty\le1}\left|\int\varphi\Xi\,d\mu\right|=\|\Xi\|_1,
\]
which is the usual duality identity between $L^1$ and $L^\infty$.
\end{proof}

For finite parameter spaces, coordinate indicators in $\QoI$ report posterior-mass variation over states. For bounded loss functions, the protocol-specific functional diameter defined in Section~\ref{sec:bounds} reports the largest variation in posterior expected loss over the selected loss class. When $\QoI$ is the unit ball of $L^\infty(\mu)$, Proposition~\ref{prop:qoi} identifies the leading coefficient with the $L^1$ coefficient of the density difference.

The finite-state construction in
Theorem~\ref{thm:witness} gives the posterior-contrast discrepancy. The following example illustrates the same second-order mechanism
in a dominated continuous model.
\begin{example}\label{ex:normal}
Let $p$ be the standard normal density on $\mathbb R$, and let $u_i=h_i/p$ be bounded smooth score perturbations with zero $p$-mean. If the perturbations are Hermite-type score perturbations multiplied by a smooth compactly supported cutoff, the assumptions of Section~\ref{sec:local} hold after renormalization. Exact renormalization may introduce second-order zero-mass coefficients, and the resulting $B_i$ should be carried in the general local chart. For equal weights and admissible nondegenerate merges, the first-order aggregate is the arithmetic mean of the score directions. Under the square-root transformed-weight balancing rule, tree dependence in a bounded posterior functional $\int \varphi q\,d\mu$ can first occur at order $\eps^2$ and is obtained by pairing $\varphi$ with the coefficient $B_T-B_{T'}$.
\end{example}

\section{Corrected Local Charts}\label{sec:normalform}

We now state a local normalization principle. It preserves the first-order statistical aggregate and removes the non-averaging second-order term by carrying the local chart coefficient together with the density approximation.

A carried coefficient state is a triple $(G,H,B)$, where $G>0$ is the transformed weight, $H$ is a zero-mass first-order density coefficient, and $B$ is a zero-mass second-order density coefficient. The corrected weighted-average update below is algebraically defined on such states. For the interpretation as a local density-chart state, we additionally require $H/p,B/p\in L^\infty(\mu)$; then $p+\eps H+\eps^2B$ is a probability density for all sufficiently small $|\eps|$. At a leaf $i$, the general state is $(G_i,h_i,B_i)$. The common special convention for purely first-order leaf perturbations is $B_i=0$. For density-chart states satisfying these bounded-ratio conditions and $J(A,B)>0$, the uncorrected merge sends two states to
\[
G_{AB}=G_A+G_B,
\quad
H_{AB}=\frac{G_AH_A+G_BH_B}{G_A+G_B}\quad\text{and}
\]
\[
B_{AB}=\frac{G_AB_A+G_BB_B}{G_A+G_B}+\tau(G_A^2,G_B^2;H_A/p,H_B/p)(H_B-H_A).
\]
The corrected merge keeps the same $G_{AB}$ and $H_{AB}$ and replaces the last term by
\[
B_{AB}^{0}=\frac{G_AB_A+G_BB_B}{G_A+G_B}.
\]
This weighted-average update is algebraically defined even when $J(A,B)=0$. Its interpretation as subtraction of the $\tau$-generated non-averaging second-order term of the original $f$-divergence balancing rule is justified here only for nondegenerate merges with $J(A,B)>0$.

\begin{theorem}\label{thm:normalform}
For any ordered tree and any leaf coefficient states $(G_i,h_i,B_i)$, $i=1,\ldots,n$, the corrected merge returns the root state
\[
\left(\sum_iG_i,\ \frac{\sum_iG_ih_i}{\sum_iG_i},\ \frac{\sum_iG_iB_i}{\sum_iG_i}\right).
\]
The corrected second-order coefficient state is independent of the aggregation tree.
\end{theorem}

\begin{proof}[Proof.]
The corrected merge is weighted averaging in the additive weights $G$. Weighted averaging is associative when the weights are carried with the values. Induction over the tree gives the stated root state. The expression depends only on the ordered leaves and their carried states.
\end{proof}

The $\tau$-generated term is the only non-averaging contribution to the binary second-order update; its elementary coboundary produces the second-order associativity defect. Removing the term gives the weighted-average root state in the carried coordinates. The correction is local and uses the expansion around $p$ and the first-order perturbation ratios. Finite-perturbation claims require separate estimates for the concrete fusion rule. For a corrected root state $(G_T,H_T,B_T^0)$ satisfying $H_T/p,B_T^0/p\in L^\infty(\mu)$, the expression
\[
p+\eps H_T+\eps^2B_T^0
\]
is the corresponding second-order density representation; it does not define a finite-$\eps$ corrected binary fusion operator.

\begin{remark}
If a concrete finite-$\eps$ corrected operator is supplied separately and, at every generated merge, admits an expansion
\[
\mathcal M_\eps^0(A,B)=p+\eps H_{AB}+\eps^2B_{AB}^0+R_\eps(A,B)\,\,\,\,\text{and}\,\,\,
\|R_\eps(A,B)\|\le C|\eps|^3,
\]
uniformly over the relevant corrected local chart, then, for fixed $n$, the remainder propagates through every ordered tree whenever weighted averaging is nonexpansive in that norm. Consequently, any two such realized tree outputs differ by $O(|\eps|^3)$. No finite-$\eps$ corrected fusion operator is constructed here.
\end{remark}

\section{Stability and Bounds}\label{sec:bounds}

We now translate the local second-order coefficients into statistical diameter statements. The first-order term is common to all admissible trees, while the second-order coefficients determine the leading density-level and quantity-of-interest diameters.

Let $\mathcal U\subseteq\BT_n$ be an admissible tree set in the sense of Theorem~\ref{thm:tree}. For the fixed divergence generator $f$, supplied weight vector $\wv$, and local input family, we define the protocol-specific density diameter by
\[
\AODfw^{\mathcal U}(p_1^\eps,\ldots,p_n^\eps)
=\sup_{T,T'\in\mathcal U}d(\Fw{T},\Fw{T'}).
\]
For a quantity-of-interest class $\QoI$ whose relevant integrals are finite, the protocol-specific functional diameter is
\[
\AODfwQ^{\mathcal U}(p_1^\eps,\ldots,p_n^\eps)
=\sup_{\varphi\in\QoI}\sup_{T,T'\in\mathcal U}
\left|\int_\Omega\varphi\{\Fw{T}-\Fw{T'}\}\,d\mu\right|.
\]
These quantities are the counterparts of Definition~\ref{def:puq} for the mass-carrying protocol. When $\mathcal U=\BT_n$ is admissible, we omit the superscript $\mathcal U$.

Let
\[
\mathcal S_{\mathcal U}=\{B_T \mid T\in\mathcal U\}
\]
be the finite set of root second-order coefficients generated by the recursion in Theorem~\ref{thm:tree}. For a norm $\|\cdot\|$ on the local chart, the second-order coefficient diameter is
\[
\Delta_2^{\mathcal U}=\sup_{B,B'\in\mathcal S_{\mathcal U}}\|B-B'\|.
\]
For a uniformly bounded class $\QoI$ of measurable functions, meaning $\sup_{\varphi\in\QoI}\|\varphi\|_\infty<\infty$, the functional coefficient diameter is
\[
\Delta_{2,\QoI}^{\mathcal U}=\sup_{\varphi\in\QoI}\sup_{B,B'\in\mathcal S_{\mathcal U}}
\left|\int_\Omega \varphi(B-B')\,d\mu\right|.
\]
Since $\mathcal U$ is finite, $\Delta_2^{\mathcal U}$ is finite. The uniform $L^\infty(\mu)$ bound on $\QoI$ ensures that $\Delta_{2,\QoI}^{\mathcal U}$ is also finite.

For the next proposition, the output distance is generated by the norm used in $\Delta_2^{\mathcal U}$, and the tree expansions hold in that norm uniformly over $T\in\mathcal U$. The quantity-of-interest class $\QoI$ is uniformly bounded, and the tree expansions hold in $L^1(\mu)$.

\begin{proposition}\label{prop:diameter-asymptotic}
\[
\AODfw^{\mathcal U}(p_1^\eps,\ldots,p_n^\eps)=\eps^2\Delta_2^{\mathcal U}+o(\eps^2),
\]
and
\[
\AODfwQ^{\mathcal U}(p_1^\eps,\ldots,p_n^\eps)=\eps^2\Delta_{2,\QoI}^{\mathcal U}+o(\eps^2).
\]
\end{proposition}

\begin{proof}[Proof.]
Theorem~\ref{thm:tree} gives $\Fw{T}-\Fw{T'}=\eps^2(B_T-B_{T'})+o(\eps^2)$ for every pair $T,T'\in\mathcal U$. Since $\mathcal U$ is finite, the remainder is uniform over all tree pairs. Taking suprema gives the density-level formula. Pairing with uniformly bounded test functions gives the quantity-of-interest formula. Appendix~\ref{app:bounds} gives the reverse-triangle derivation.
\end{proof}

The next result expresses the coefficient diameter through elementary rotations. It is the second-order counterpart of Proposition~\ref{prop:assoc}. For a rotation $\rho$ along a tree path, let $\Curv_\rho$ be the local coefficient from Proposition~\ref{prop:rotation} and let $L_\rho$ be the Lipschitz modulus of the surrounding first-order averaging context on second-order coefficients. In the square-root chart, $L_\rho\le1$ for the $L^1$ norm, the signed-measure total variation norm, the coefficient norm $\frac12\|\cdot\|_1$ associated with probability total-variation distance, and the full $L^\infty$-unit-ball dual norm, because the surrounding context is convex averaging.

\begin{theorem}\label{thm:rotation-bound}
For any two trees $T,T'\in\mathcal U$ and any rotation path $\gamma:T\leadsto T'$ whose intermediate trees all belong to $\mathcal U$,
\[
\|B_T-B_{T'}\|\le \sum_{\rho\in\gamma}L_\rho\|\Curv_\rho\|.
\]
The coefficient diameter satisfies
\[
\Delta_2^{\mathcal U}\le \sup_{T,T'\in\mathcal U}\inf_{\gamma}\sum_{\rho\in\gamma}L_\rho\|\Curv_\rho\|,
\]
where the infimum is taken over rotation paths from $T$ to $T'$ whose intermediate trees all belong to $\mathcal U$, with the convention that the infimum over an empty path set is $+\infty$.
\end{theorem}

\begin{proof}[Proof.]
An elementary rotation changes the second-order coefficient by inserting the local coefficient $\Curv_\rho$ inside the surrounding averaging context. The context has Lipschitz modulus $L_\rho$, so the propagated change has norm at most $L_\rho\|\Curv_\rho\|$. Summing along the finite path gives the first inequality. Taking the infimum over paths and then the supremum over endpoint pairs gives the second inequality.
\end{proof}

The following corollary gives a compact diagnostic in terms of a maximal local rotation coefficient and a path-length bound. Let $L_\rho\le L$ and $\|\Curv_\rho\|\le c$ for every rotation reached by the generated local chart, and let every pair of trees in $\mathcal U$ be joined inside $\mathcal U$ by a rotation path of length at most $m_n$.

\begin{corollary}\label{cor:global-bound}
The restricted density-level diameter satisfies
\[
\AODfw^{\mathcal U}(p_1^\eps,\ldots,p_n^\eps)\le \eps^2Lm_nc+o(\eps^2).
\]
\end{corollary}

\begin{proof}[Proof.]
Theorem~\ref{thm:rotation-bound} gives $\Delta_2^{\mathcal U}\le Lm_nc$. Proposition~\ref{prop:diameter-asymptotic} gives the stated bound.
\end{proof}

We next state a complementary lower-bound principle for a nonzero rotation coefficient. Assume the local metric expansion in Proposition~\ref{prop:rotation}. Let $\rho:T\to T'$ be an elementary rotation edge with $T,T'\in\mathcal U$, and let
\[
K_\rho=B_{T'}-B_T
\]
be its context-propagated second-order coefficient.

\begin{theorem}\label{thm:lower}
If $K_\rho\ne0$, then the corresponding local input family satisfies
\[
\liminf_{\eps\downarrow0}\eps^{-2}\AODfw^{\mathcal U}(p_1^\eps,\ldots,p_n^\eps)
\ge \|K_\rho\|.
\]
For a root rotation from $(AB)C$ to $A(BC)$, $K_\rho=-\Curv(A,B,C)$; for the reverse orientation, $K_\rho=\Curv(A,B,C)$. In either case, the lower bound uses $\|\Curv(A,B,C)\|$.
\end{theorem}

\begin{proof}[Proof of Theorem~\ref{thm:lower}.]
The restricted diameter is the supremum over all tree pairs in $\mathcal U$, so it is at least the distance between the endpoints of $\rho$. The local edge expansion gives $\Fw{T'}-\Fw{T}=\eps^2K_\rho+o(\eps^2)$. The metric expansion then gives $d(\Fw{T},\Fw{T'})=\eps^2\|K_\rho\|+o(\eps^2)$. Division by $\eps^2$ and passage to the lower limit prove the claim.
\end{proof}

A scalar calibration assigns each leaf a positive transformed mass $G_i=g(w_i)$, propagates masses additively, and uses endpoint balancing weights proportional to $G_A^2$ and $G_B^2$. Equal supplied weights give equal transformed leaf masses. The next statement shows that such a scalar weight change cannot remove the finite three-state obstruction.

\begin{proposition}\label{prop:scalar}
Assume $f'''(1)\ne0$. For the equal-weight three-point witness of Theorem~\ref{thm:witness}, every scalar calibration described above leaves the second-order rotation coefficient equal to
\[
-\frac{3\alpha^2}{4}\frac{f'''(1)}{f''(1)}(1,0,-1).
\]
Scalar weight calibration alone does not make the second-order root coefficient independent of the aggregation tree on that local input family.
\end{proposition}

\begin{proof}[Proof.]
Let the common supplied weight be $w_0$, and let $G_0=g(w_0)>0$ be the common transformed leaf mass. The lower merge has transformed mass $2G_0$, and the upper merge has endpoint transformed masses $2G_0$ and $G_0$. The endpoint balancing weights are proportional to $4G_0^2$ and $G_0^2$, so their common factor cancels in $t_0$ and in $\tau$. The computation from Theorem~\ref{thm:witness} is unchanged and gives the stated nonzero vector.
\end{proof}

\section{Conclusion}\label{sec:conclusion}
We have studied a local asymptotic theory of aggregation-order variation for density fusion under a fixed mass-carrying divergence-balancing protocol. The protocol-specific diameter retains the supplied weights and captures the spread of final density summaries obtained by changing only the ordered aggregation tree. For smooth divergence balancing on an admissible fixed-size tree set, the expansion gives a precise separation: square-root-transformed weights remove first-order tree dependence, while the first potentially tree-dependent coefficient appears at second order.

We have shown that the second-order coefficient has a direct statistical interpretation. It is governed by the local curvature ratio $f'''(1)/f''(1)$ and by a cubic moment of first-order perturbation contrasts. The finite-state posterior calculation shows that an ordinary posterior contrast can detect the coefficient. The quantity-of-interest result transfers the density-level expansion to bounded posterior means, tail probabilities, bounded losses, and finite-state contrasts whenever the selected functional class detects the signed coefficient. For unbounded posterior means, the same transfer requires additional uniform moment or weighted-$L^1$ assumptions.

When the protocol carries the required chart data, our local correction cancels the tree-dependent second-order coefficient. The same coefficient therefore diagnoses local tree sensitivity and specifies the correction required to remove it. We leave three directions for future work: deriving nonasymptotic bounds for concrete posterior-fusion rules, extending the correction to particle and mixture representations, and studying data-dependent tree selection in distributed Bayesian computation.

\vspace{1cm}
\bibliographystyle{splncs04}
\bibliography{ref}

\newpage
\appendix
\section*{Appendix}
\addcontentsline{toc}{section}{Appendix: Extended Proofs and Auxiliary Results}

Appendices~A--D give the extended derivations for Sections~4--7. Appendix~E gives the complete derivations for the stability, diameter, rotation-bound, lower-bound, and calibration results in Section~\ref{sec:bounds}.

\section{Proofs for Section 4}
 The reference density $p$ is strictly positive, and all density ratios used below remain in a bounded neighborhood of one. The ratios $u_A=h_A/p$, $u_B=h_B/p$, $b_A=B_A/p$, and $b_B=B_B/p$ have zero $p$-mean because the corresponding density coefficients have zero $\mu$-integral. We use
\[
\Delta=u_B-u_A,
\quad
p_t^\eps=(1-t)r^\eps+ts^\eps,
\quad
\nu_t=u_A+t\Delta
\,\,\,\,\text{and}\,\,\,\,\,
c_t=b_A+t(b_B-b_A).
\]
The bounded-ratio hypotheses give all Taylor remainders below uniformly for $t$ in each compact subinterval of $[0,1]$ and, in the root proof, uniformly on a fixed neighborhood of $t_0$.

\subsection{Local divergence expansion}

We first derive the expansion of the two divergences used in the balancing equation. The endpoint densities have the form
\[
r^\eps=p\{1+\eps u_A+\eps^2b_A+o(\eps^2)\}
\,\,\,\,\text{and}\,\,\,\,\,
s^\eps=p\{1+\eps u_B+\eps^2b_B+o(\eps^2)\}.
\]
Consequently,
\[
\frac{p_t^\eps}{p}=1+\eps\nu_t+\eps^2c_t+o(\eps^2).
\]
For any bounded $x_\eps=\eps\nu_t+\eps^2c_t+o(\eps^2)$, the identity $(1+x_\eps)^{-1}=1-x_\eps+x_\eps^2+o(\eps^2)$ gives
\[
\frac{p}{p_t^\eps}=1-\eps\nu_t+
\eps^2(\nu_t^2-c_t)+o(\eps^2).
\]
Multiplication with $r^\eps/p$ yields
\[
\begin{aligned}
\frac{r^\eps}{p_t^\eps}
&=(1+\eps u_A+\eps^2b_A)
  (1-\eps\nu_t+\eps^2(\nu_t^2-c_t))+o(\eps^2)\\
&=1+\eps(u_A-\nu_t)+\eps^2\{b_A+\nu_t^2-c_t-u_A\nu_t\}+o(\eps^2).
\end{aligned}
\]
Since $u_A-\nu_t=-t\Delta$ and $c_t=b_A+t(b_B-b_A)$, the second-order coefficient becomes
\[
b_A+\nu_t^2-c_t-u_A\nu_t
=t\{\Delta(u_A+t\Delta)-(b_B-b_A)\}.
\]
Hence,
\[
\frac{r^\eps}{p_t^\eps}=1-\eps t\Delta+
\eps^2t\{\Delta(u_A+t\Delta)-(b_B-b_A)\}+o(\eps^2).
\]
The same calculation for $s^\eps$ gives
\[
\frac{s^\eps}{p_t^\eps}=1+\eps(1-t)\Delta+
\eps^2(1-t)\{(b_B-b_A)-\Delta(u_A+t\Delta)\}+o(\eps^2).
\]
Let
\[
A_2=\frac{f''(1)}2 \,\,\,\,\text{and}\,\,\,\,\,
A_3=\frac{f'''(1)}6.
\]
Because $f(1)=f'(1)=0$ after the standard affine normalization of an $f$-divergence generator, Taylor's formula gives
\[
f(1+z)=A_2z^2+A_3z^3+o(z^3)
\]
uniformly for bounded $z\to0$. We apply the formula with
\[
z_r=-\eps t\Delta+
\eps^2t\{\Delta\nu_t-(b_B-b_A)\}+o(\eps^2).
\]
Then
\[
z_r^2=\eps^2t^2\Delta^2-2\eps^3t^2\Delta\{\Delta\nu_t-(b_B-b_A)\}+o(\eps^3)
\]
and
\[
z_r^3=-\eps^3t^3\Delta^3+o(\eps^3).
\]
Multiplying by the middle density
\[
p_t^\eps=p\{1+\eps\nu_t+O(\eps^2)\}
\]
and retaining terms through order $\eps^3$ gives
\[
\begin{aligned}
D_f(r^\eps\|p_t^\eps)
&=\int p_t^\eps f\!\left(\frac{r^\eps}{p_t^\eps}\right)d\mu\\
&=A_2\eps^2t^2\int\Delta^2p\,d\mu\\
&\quad+\eps^3\Biggl\{
A_2t^2\int\{-2\Delta(\Delta\nu_t-(b_B-b_A))
+\Delta^2\nu_t\}p\,d\mu
-A_3t^3\int\Delta^3p\,d\mu\Biggr\}+o(\eps^3).
\end{aligned}
\]
The integrand in the $A_2$ term reduces to
\[
-2\Delta^2\nu_t+2\Delta(b_B-b_A)+\Delta^2\nu_t
=2\Delta(b_B-b_A)-\Delta^2\nu_t.
\]
With
\[
\begin{aligned}
J&=\int\Delta^2p\,d\mu,
&L_3&=\int\Delta^3p\,d\mu,\\
M&=\int\Delta(b_B-b_A)p\,d\mu \quad\text{and}\quad
&K(t)&=\int\Delta^2\nu_t p\,d\mu.
\end{aligned}
\]
we obtain
\[
D_f(r^\eps\|p_t^\eps)=A_2\eps^2t^2J+
\eps^3\{A_2t^2(2M-K(t))-A_3t^3L_3\}+o(\eps^3).
\]
For the second endpoint, we set
\[
z_s=\eps(1-t)\Delta+
\eps^2(1-t)\{(b_B-b_A)-\Delta\nu_t\}+o(\eps^2).
\]
Then,
\[
z_s^2=\eps^2(1-t)^2\Delta^2+2\eps^3(1-t)^2\Delta\{(b_B-b_A)-\Delta\nu_t\}+o(\eps^3)
\]
and
\[
z_s^3=\eps^3(1-t)^3\Delta^3+o(\eps^3).
\]
Multiplication by $p_t^\eps$ gives
\[
D_f(s^\eps\|p_t^\eps)=A_2\eps^2(1-t)^2J+
\eps^3\{A_2(1-t)^2(2M-K(t))+A_3(1-t)^3L_3\}+o(\eps^3).
\]
These two terms are the local divergence expansions used in Section~\ref{sec:local}.

\subsection{Root selection and uniformity}

The balancing residual is
\[
\Phi_\eps(t)=aD_f(r^\eps\|p_t^\eps)-bD_f(s^\eps\|p_t^\eps).
\]
The previous subsection gives
\[
\eps^{-2}\Phi_\eps(t)=\Phi_2(t)+\eps\Phi_3(t)+o(\eps),
\]
where $\Phi_2(t)=A_2J\{at^2-b(1-t)^2\}$.
The point
\[
t_0=\frac{\sqrt b}{\sqrt a+\sqrt b}
\]
is the unique zero of $\Phi_2$ in $(0,1)$ because $at^2=b(1-t)^2$ and the positive square root gives $\sqrt a\,t=\sqrt b(1-t)$. Moreover,
\[
\Phi_2'(t)=2A_2J\{at+b(1-t)\} \,\,\,\,\text{and}\,\,\,\,\,
\Phi_2'(t_0)=2A_2J\{at_0+b(1-t_0)\}>0.
\]

The derivative statement used for uniqueness follows from the exact integrands. Let
\[
R_t^\eps=\frac{r^\eps}{p_t^\eps},
\quad
S_t^\eps=\frac{s^\eps}{p_t^\eps}
\quad\text{and}\quad
\partial_t p_t^\eps=s^\eps-r^\eps.
\]
Direct differentiation gives
\[
\partial_t R_t^\eps
=-R_t^\eps\frac{s^\eps-r^\eps}{p_t^\eps}
\quad\text{and}\quad
\partial_t S_t^\eps
=-S_t^\eps\frac{s^\eps-r^\eps}{p_t^\eps}.
\]
Consequently,
\[
\begin{aligned}
\partial_t\{p_t^\eps f(R_t^\eps)\}
&=(s^\eps-r^\eps)\{f(R_t^\eps)-R_t^\eps f'(R_t^\eps)\},\\
\partial_t\{p_t^\eps f(S_t^\eps)\}
&=(s^\eps-r^\eps)\{f(S_t^\eps)-S_t^\eps f'(S_t^\eps)\}.
\end{aligned}
\]
For $g(x)=f(x)-xf'(x)$, Taylor expansion at one yields
\[
g(1+z)=-2A_2z+O(z^2).
\]

The derivative Taylor expansion is
\[
f'(1+z)=2A_2z+3A_3z^2+o(z^2)
=f''(1)z+\frac12f'''(1)z^2+o(z^2).
\]
Lemma~\ref{lem:chart} gives $R_t^\eps-1=o(1)$ and $S_t^\eps-1=o(1)$ uniformly for $|t-t_0|\le\eta$. Bounded endpoint-to-reference ratios imply
\[
|s^\eps-r^\eps|\le C|\eps|p,
\]
uniformly on the same neighborhood. More precisely, let
\[
\delta_\eps=sup_{|t-t_0|\le\eta}
\max\{\|R_t^\eps-1\|_\infty,\|S_t^\eps-1\|_\infty\}=O(|\eps|).
\]
There is a function $\rho(\delta)\to0$ as $\delta\downarrow0$ such that
\[
|f'(1+z)-2A_2z-3A_3z^2|\le \rho(\delta)z^2,
\qquad |z|\le\delta.
\]
In particular, $|g(1+z)+2A_2z|\le Cz^2$ on a fixed neighborhood of zero. Hence, the Taylor-remainder contribution to $\partial_t\Phi_\eps(t)$ is bounded uniformly in $t$ by
\[
C|\eps|p\,\delta_\eps^2=O(|\eps|^3)p=o(\eps^2)p.
\]
Integration against $\mu$ gives the uniform $o(\eps^2)$ derivative remainder below.

The ratio expansions in the preceding subsection and the endpoint difference give, uniformly for $|t-t_0|\le\eta$,
\begin{itemize}[label=\textendash]
    \item $R_t^\eps-1=-\eps t\Delta+O(\eps^2),$
\item $S_t^\eps-1=\eps(1-t)\Delta+O(\eps^2),$
\item $s^\eps-r^\eps=\eps p\Delta+O(\eps^2p).$
\end{itemize}

It follows that
\[
\begin{aligned}
\partial_t\Phi_\eps(t)
&=\int (s^\eps-r^\eps)
   \{a g(R_t^\eps)-b g(S_t^\eps)\}\,d\mu\\
&=2A_2\eps^2\{at+b(1-t)\}
  \int\Delta^2p\,d\mu+o(\eps^2)\\
&=\eps^2\Phi_2'(t)+o(\eps^2).
\end{aligned}
\]
Therefore,
\[
\eps^{-2}\partial_t\Phi_\eps(t)=\Phi_2'(t)+o(1)
\]
uniformly on the fixed neighborhood.

We now prove existence and uniqueness. Since $\Phi_2'(t_0)>0$, there are $c>0$ and $\eta>0$ with $\eta<\min\{t_0,1-t_0\}$ such that $\Phi_2'(t)\ge2c$ whenever $|t-t_0|\le\eta$. Uniform derivative convergence implies $\partial_t\Phi_\eps(t)\ge c\eps^2>0$ on the interval for all sufficiently small nonzero $\eps$. Hence, $\Phi_\eps$ is strictly increasing on the interval. Also, the signs of $\Phi_2(t_0-\eta)$ and $\Phi_2(t_0+\eta)$ are negative and positive, respectively. Uniform convergence of $\eps^{-2}\Phi_\eps$ to $\Phi_2$ preserves these signs for small $|\eps|$. The intermediate value theorem gives a root in $(t_0-\eta,t_0+\eta)$, and strict monotonicity gives uniqueness of that local root.

Let $t_\eps$ denote the unique root. The sign argument gives $t_\eps\to t_0$. Dividing the root equation by $\eps^2$ gives
\[
0=\Phi_2(t_\eps)+\eps\Phi_3(t_\eps)+o(\eps).
\]
Since $\Phi_2$ has a simple zero at $t_0$, there are $c_0>0$ and a neighborhood $I$ of $t_0$ such that $|\Phi_2(t)|\ge c_0|t-t_0|$ for $t\in I$. 
Local boundedness of $\Phi_3$ gives a constant $C>0$ such that
\[
\begin{aligned}
c_0|t_\eps-t_0|
&\le |\Phi_2(t_\eps)|\\
&\le |\eps|\,|\Phi_3(t_\eps)|+o(|\eps|)\\
&\le C|\eps|+o(|\eps|).
\end{aligned}
\]
Division by $c_0$ proves $t_\eps-t_0=O(\eps)$.

Taylor expansion of $\Phi_2$ at $t_0$ and continuity of the coefficient $\Phi_3$ give
\[
\Phi_2(t_\eps)=\Phi_2'(t_0)(t_\eps-t_0)+o(\eps)
\quad\text{and}\quad
\Phi_3(t_\eps)=\Phi_3(t_0)+o(1).
\]
Substitution in the root equation divided by $\eps^2$ yields
\[
0=\Phi_2'(t_0)(t_\eps-t_0)+\eps\Phi_3(t_0)+o(\eps).
\]
After division by $\eps$,
\[
\frac{t_\eps-t_0}{\eps}=-\frac{\Phi_3(t_0)}{\Phi_2'(t_0)}+o(1).
\]
Thus, $t_\eps=t_0+\eps t_1+o(\eps)$ with $t_1=-\Phi_3(t_0)/\Phi_2'(t_0)$.
At $t_0$, the equality $at_0^2=b(1-t_0)^2$ cancels the shared $A_2(2M-K(t_0))$ term in $\Phi_3(t_0)$. Hence,
\[
\Phi_3(t_0)=-A_3L_3\{at_0^3+b(1-t_0)^3\}.
\]
Consequently,
\[
t_1=\frac{A_3}{\Phi_2'(t_0)}L_3\{at_0^3+b(1-t_0)^3\}.
\]
Using $A_3/A_2=f'''(1)/(3f''(1))$ and $\Phi_2'(t_0)=2A_2J\{at_0+b(1-t_0)\}$ gives the coefficient $\tau$ used in Theorem~\ref{thm:binary}.

For completeness, let $\mathcal N_k$ be the finite set of nodes of height at most $k$ occurring in the trees of $\mathcal U$, and let $r_v^\eps$ denote the remainder at a node $v$. Set
\[
R_k(\eps)=\max_{v\in\mathcal N_k}
\frac{\|r_v^\eps\|_{p,\infty}}{\eps^2}.
\]
At height zero, finiteness of the leaf set and the leaf assumptions give
\[
R_0(\eps)\longrightarrow0.
\]
Assume $R_{k-1}(\eps)\to0$. Every node $v\in\mathcal N_k\setminus\mathcal N_{k-1}$ merges two nodes of height at most $k-1$. The explicit parent-remainder identity derived in Appendix~B expresses $r_v^\eps$ as a convex combination of the two child remainders plus terms of order
\[
\eps\,o(\eps),
\quad
O(\eps)O(\eps^2)
\quad\text{and}\quad
O(\eps)o(\eps^2).
\]
After division by $\eps^2$, every additional term tends to zero. Since $\mathcal N_k\setminus\mathcal N_{k-1}$ is finite,
\[
R_k(\eps)
\le R_{k-1}(\eps)+o(1)
\longrightarrow0.
\]
Induction up to height $n-1$ gives
\[
\max_{T\in\mathcal U}
\frac{\|r_T^\eps\|_{p,\infty}}{\eps^2}
\le R_{n-1}(\eps)\longrightarrow0,
\]
which proves the uniform root remainder.

\section{Proofs for Section 5}

We give complete derivations for the tree expansion and for the mass constraints implicit in the density chart. The proof uses structural induction over ordered binary trees.

\subsection{Mass preservation}

At every leaf $i$, the local density expansion
\[
p_i^\eps=p+\eps h_i+\eps^2B_i+o(\eps^2)
\]
consists of densities integrating to one. Integrating both sides and comparing the coefficients of $\eps$ and $\eps^2$ gives
\[
\int h_i\,d\mu=0 \,\,\text{and}\,\,
\int B_i\,d\mu=0.
\]
We prove by induction that every generated node has zero-mass first- and second-order coefficients. The assertion is true at leaves by the preceding term. For an internal merge of child states $(G_A,H_A,B_A)$ and $(G_B,H_B,B_B)$, let
$\tau_{AB}=\tau(G_A^2,G_B^2;H_A/p,H_B/p)$. The first-order coefficient is
\[
H_{AB}=\frac{G_AH_A+G_BH_B}{G_A+G_B}.
\]
If $\int H_A\,d\mu=\int H_B\,d\mu=0$, then $\int H_{AB}\,d\mu=0$. The second-order coefficient is
\[
B_{AB}=\frac{G_AB_A+G_BB_B}{G_A+G_B}+\tau_{AB}(H_B-H_A).
\]
The scalar $\tau_{AB}$ may depend on the child perturbation ratios, but it is constant with respect to the integration variable. Therefore,
\[
\int B_{AB}d\mu=\frac{G_A\int B_Ad\mu+G_B\int B_Bd\mu}{G_A+G_B}+\tau_{AB}\left(\int H_Bd\mu-\int H_Ad\mu\right)=0.
\]
The corrected recursion removes the $\tau_{AB}(H_B-H_A)$ term and is a weighted average, so the same induction proves zero mass for the corrected second-order coefficients.

\subsection{Proof of the tree expansion}

We prove Theorem~\ref{thm:tree} by structural induction. For a node $A$, let $(G_A,q_A^\eps)$ be the state assigned to the subtree rooted at $A$ by the mass-carrying protocol. The induction invariant states that the density component has an expansion
\[
q_A^\eps=p+\eps H_A+\eps^2B_A+r_A^\eps
\,\,\,\,\text{and}\,\,\,\,
\|r_A^\eps\|_{p,\infty}=o(\eps^2),
\]
uniformly over all generated subtrees with the same leaf set size bounded by $n$, and the associated transformed mass is
\[
G_A=\sum_{i\in A}G_i.
\]
At a leaf $A=\{i\}$, the assertion is precisely the leaf assumption with $q_A^\eps=p_i^\eps$, $G_A=G_i$, $H_A=h_i$, and the given $B_A=B_i$. In the special purely first-order leaf convention, $B_i=0$.

Assume the invariant holds for the two children $A$ and $B$ of an internal node. Their endpoint expansions are
\[
q_A^\eps=p+\eps H_A+\eps^2B_A+r_A^\eps
\,\,\,\,\text{and}\,\,\,\,
q_B^\eps=p+\eps H_B+\eps^2B_B+r_B^\eps,
\]
with $r_A^\eps=o(\eps^2)$ and $r_B^\eps=o(\eps^2)$ uniformly. The binary balancing theorem applies to these endpoint summaries with balancing weights $G_A^2$ and $G_B^2$. The leading root weight for the second endpoint is
\[
t_0=\frac{\sqrt{G_B^2}}{\sqrt{G_A^2}+\sqrt{G_B^2}}=\frac{G_B}{G_A+G_B}.
\]
Admissibility supplies $H_A/p,H_B/p,B_A/p,B_B/p\in L^\infty(\mu)$ and $J(A,B)>0$ for the generated merge, so Lemma~\ref{lem:root} selects $t_{AB}^\eps$ and Theorem~\ref{thm:binary} applies at the induction step.
The first-order coefficient at the parent is therefore
\[
(1-t_0)H_A+t_0H_B
=\frac{G_AH_A+G_BH_B}{G_A+G_B}.
\]
The second-order coefficient has two sources: the weighted average of the child second-order coefficients and the local shift of the balancing root. Hence,
\[
B_{AB}=\frac{G_AB_A+G_BB_B}{G_A+G_B}
+\tau(G_A^2,G_B^2;H_A/p,H_B/p)(H_B-H_A).
\]
The transformed mass at the parent is additive:
\[
G_{AB}=G_A+G_B.
\]

We verify the parent remainder explicitly. Let
\[
\tau_{AB}=\tau(G_A^2,G_B^2;H_A/p,H_B/p)
\]
and write the balancing root as
\[
t_\eps=t_0+\eps\tau_{AB}+\zeta_{AB}^\eps
\quad\text{and}\quad
\zeta_{AB}^\eps=o(\eps).
\]
Since the parent density equals
\[
q_{AB}^\eps
=(1-t_0)q_A^\eps+t_0q_B^\eps
+(\eps\tau_{AB}+\zeta_{AB}^\eps)
 (q_B^\eps-q_A^\eps),
\]
substitution of the two child expansions gives
\[
q_{AB}^\eps
=p+\eps H_{AB}+\eps^2B_{AB}+r_{AB}^\eps,
\]
where
\[
\begin{aligned}
r_{AB}^\eps
&=(1-t_0)r_A^\eps+t_0r_B^\eps
  +\eps\zeta_{AB}^\eps(H_B-H_A)\\
&\quad+\eps^2(\eps\tau_{AB}+\zeta_{AB}^\eps)(B_B-B_A)\\
&\quad+(\eps\tau_{AB}+\zeta_{AB}^\eps)(r_B^\eps-r_A^\eps).
\end{aligned}
\]
The four terms satisfy, respectively,
\[
o(\eps^2),
\quad
\eps o(\eps)=o(\eps^2),
\quad
O(\eps^3)
\quad\text{and}\quad
O(\eps)o(\eps^2)=o(\eps^2).
\]
Therefore, $\|r_{AB}^\eps\|_{p,\infty}=o(\eps^2)$, and the induction invariant holds at the parent.

The root first-order coefficient is independent of the tree because the recursion for $H_A$ is the weighted average associated with the additive masses $G_A$. A second induction over subtrees gives
\[
H_A=\frac{\sum_{i\in A}G_ih_i}{\sum_{i\in A}G_i}.
\]
At the root $A=\{1,\ldots,n\}$, the formula becomes
\[H=\frac{\sum_iG_ih_i}{\sum_iG_i}.\]
The protocol definition gives $q_{\{1,\ldots,n\}}^\eps=\Fw{T}$ at the root. The second-order coefficient is exactly the value obtained by applying the stated recursion at every internal node of the chosen tree. Since the binary remainder is uniform on the common bounded chart neighborhood and the number of internal nodes is finite, the induction on node height in Appendix~A gives a root remainder $o(\eps^2)$ uniformly over $T\in\mathcal U$.

\section{Proofs for Section 6}

We expand the finite-state witness and the quantity-of-interest argument. The finite witness uses only three states and equal supplied weights, so every coefficient can be checked directly.

\subsection{Finite-state witness}

Let the reference density on three states be $p=(1/3,1/3,1/3)$. The perturbations are
\[
h_1=\alpha(2,-1,-1),\quad h_2=\alpha(-1,2,-1) \,\,\,\,\text{and}\,\,\,\,\, h_3=\alpha(-1,-1,2).
\]
Each perturbation has zero sum. For sufficiently small $|\eps|$, the leaf densities $p_i^\eps=p+\eps h_i$ are strictly positive and have total mass one. The supplied weights are equal, so the transformed leaf masses are $G_1=G_2=G_3=1$.

We first compute the left bracketing $((12)3)$. At the lower merge,
\[
H_{12}=\frac{h_1+h_2}{2}=\frac{\alpha}{2}(1,1,-2),
\,\,\text{and}\,\,
G_{12}=2.
\]

The lower contrast and cubic moment are
\[
\Delta_{1,2}=\frac{h_2-h_1}{p}=9\alpha(-1,1,0)
\]
and
\[
L_3(1,2)=\frac13\{(-9\alpha)^3+(9\alpha)^3+0^3\}=0.
\]
Let $\tau_{1,2}=\tau(G_1^2,G_2^2;h_1/p,h_2/p)$. Thus, $\tau_{1,2}=0$, and the lower second-order coefficient is
\[
B_{12}=\tau_{1,2}(h_2-h_1)=0.
\]
The upper merge combines block $12$ with leaf $3$, so
\[
H_3-H_{12}=\alpha(-1,-1,2)-\frac{\alpha}{2}(1,1,-2)=\frac{3\alpha}{2}(-1,-1,2).
\]
Dividing by $p=(1/3,1/3,1/3)$ gives the first-order ratio contrast
\[
\Delta_{(12),3}=\frac{H_3-H_{12}}p=\frac{9\alpha}{2}(-1,-1,2).
\]
For the upper merge, the weights are $a=G_{12}^2=4$ and $b=G_3^2=1$, so
\[
t_0=\frac{1}{2+1}=\frac13.
\]
The cubic moment is
\[
\begin{aligned}
L_3&=\int \Delta_{(12),3}^3p\,d\mu
=\frac13\left[\left(-\frac{9\alpha}{2}\right)^3+
\left(-\frac{9\alpha}{2}\right)^3+(9\alpha)^3\right]\\
&=\frac13\left[-\frac{729\alpha^3}{8}-\frac{729\alpha^3}{8}+729\alpha^3\right]
=\frac{729\alpha^3}{4}.
\end{aligned}
\]
The quadratic moment is
\[
J=\int \Delta_{(12),3}^2p\,d\mu
=\frac13\left[\frac{81\alpha^2}{4}+\frac{81\alpha^2}{4}+81\alpha^2\right]
=\frac{81\alpha^2}{2}.
\]
Consequently, $L_3/J=9\alpha/2$. Let $\tau_{(12),3}=\tau(G_{12}^2,G_3^2;H_{12}/p,H_3/p)$. Substitution into the expression for $\tau$ with $a=4$, $b=1$, and $t_0=1/3$ gives
\[
\tau_{(12),3}
=\frac{f'''(1)}{3f''(1)}\cdot\frac{\alpha}{2}=\frac{k\alpha}{2}
\,\,\,\,\text{and}\,\,\,\,\,
k=\frac{f'''(1)}{3f''(1)}.
\]
The upper contribution is
\[
B^{(12)3}=\tau_{(12),3}(H_3-H_{12})
=\frac{k\alpha}{2}\cdot\frac{3\alpha}{2}(-1,-1,2)
=k\alpha^2\left(-\frac34,-\frac34,\frac32\right).
\]

We now compute the right bracketing $1(23)$. At the lower merge,
\[
H_{23}=\frac{h_2+h_3}{2}=\frac{\alpha}{2}(-2,1,1),
\,\,\text{and}\,\,
G_{23}=2.
\]

At the lower merge,
\[
\Delta_{2,3}=\frac{h_3-h_2}{p}=9\alpha(0,-1,1),
\]
and therefore
\[
L_3(2,3)=\frac13\{0^3+(-9\alpha)^3+(9\alpha)^3\}=0.
\]
Let $\tau_{2,3}=\tau(G_2^2,G_3^2;h_2/p,h_3/p)$. Consequently, $\tau_{2,3}=0$ and $B_{23}=0$.

The upper contrast is
\[
H_{23}-H_1=\frac{\alpha}{2}(-2,1,1)-\alpha(2,-1,-1)=\frac{3\alpha}{2}(-2,1,1).
\]

The corresponding ratio contrast is
\[
\Delta_{1,(23)}=\frac{H_{23}-H_1}{p}=\frac{9\alpha}{2}(-2,1,1).
\]
Its quadratic and cubic moments are
\[
\begin{aligned}
J_{1,(23)}
&=\frac13\left\{81\alpha^2+\frac{81\alpha^2}{4}
                   +\frac{81\alpha^2}{4}\right\}
=\frac{81\alpha^2}{2},\\
L_{3,1,(23)}
&=\frac13\left\{-729\alpha^3+\frac{729\alpha^3}{8}
                    +\frac{729\alpha^3}{8}\right\}
=-\frac{729\alpha^3}{4}.
\end{aligned}
\]
Hence, $L_{3,1,(23)}/J_{1,(23)}=-9\alpha/2$. The upper weights are
\[
a=G_1^2=1,
\quad
b=G_{23}^2=4
\quad\text{and}\quad
t_0=\frac23.
\]
Let $\tau_{1,(23)}=\tau(G_1^2,G_{23}^2;H_1/p,H_{23}/p)$. Substitution in Theorem~\ref{thm:binary} gives
\[
\begin{aligned}
\tau_{1,(23)}
&=\frac{A_3}{A_2}\frac{L_{3,1,(23)}}{J_{1,(23)}}
  \frac{1(2/3)^3+4(1/3)^3}
       {2\{1(2/3)+4(1/3)\}}\\
&=k\left(-\frac{9\alpha}{2}\right)\frac19
=-\frac{k\alpha}{2}.
\end{aligned}
\]
Therefore,
\[
B^{1(23)}
=\tau_{1,(23)}(H_{23}-H_1)
=-\frac{k\alpha}{2}\frac{3\alpha}{2}(-2,1,1)
=k\alpha^2\left(\frac32,-\frac34,-\frac34\right).
\]
Therefore,
\[
B^{(12)3}-B^{1(23)}
=k\alpha^2\left(-\frac94,0,\frac94\right)
=-\frac{9k\alpha^2}{4}(1,0,-1)
=-\frac{3\alpha^2}{4}\frac{f'''(1)}{f''(1)}(1,0,-1).
\]
For the centered Kullback--Leibler generator $f(u)=u\log u-u+1$, $f''(1)=1$ and $f'''(1)=-1$. Hence, the leading difference is
\[
\frac{3\alpha^2}{4}(1,0,-1).
\]
Testing against the contrast $(1,0,-1)$ gives the leading coefficient
\[
\left\langle (1,0,-1),\frac{3\alpha^2}{4}(1,0,-1)\right\rangle=\frac{3\alpha^2}{2}
\]
under counting measure.

\subsection{Proof of Proposition~\ref{prop:qoi}}

Let $T$ and $T'$ be two trees for which
\[
\Fw{T}-\Fw{T'}=\eps^2\Xi+r^\eps,
\qquad
\|r^\eps\|_1=o(\eps^2).
\]
For a bounded test function $\varphi$, H\"older's inequality gives
\[
\left|\int \varphi r^\eps\,d\mu\right|
\le \|\varphi\|_\infty\|r^\eps\|_1=o(\eps^2).
\]
Therefore,
\[
\int \varphi(\Fw{T}-\Fw{T'})d\mu
=\eps^2\int\varphi\Xi\,d\mu+o(\eps^2).
\]
If $\int\varphi\Xi\,d\mu\ne0$, division by $\eps^2$ and passage to the limit gives the claimed nonzero leading coefficient. Taking the absolute value yields the same leading magnitude because the nonzero limit has a fixed sign for all sufficiently small positive $\eps$.

For the full bounded unit ball, let $\QoI=\{\varphi \mid \|\varphi\|_\infty\le1\}$. The inequality
\[
\sup_{\|\varphi\|_\infty\le1}\left|\int\varphi\Xi\,d\mu\right|\le \|\Xi\|_1
\]
follows from H\"older's inequality. Conversely, if $\Xi$ is integrable, the measurable function $\operatorname{sgn}(\Xi)$ satisfies $\|\operatorname{sgn}(\Xi)\|_\infty\le1$ and gives
\[
\int \operatorname{sgn}(\Xi)\Xi\,d\mu=\|\Xi\|_1.
\]
Thus, the supremum equals $\|\Xi\|_1$.

\section{Proofs for Section 7}
 The corrected chart is an asymptotic representation. It need not coincide with the finite-$\eps$ output of the original divergence-balancing operation.

\subsection{\texorpdfstring{Affine normal forms and rotation coboundary}{Affine normal forms and rotation coboundary}}

A carried coefficient state is a triple $(G,H,B)$ with $G>0$ and with $H$ and $B$ zero-mass density coefficients. The corrected merge is algebraically defined on such states. When the state is interpreted in the local density chart, we additionally require $H/p,B/p\in L^\infty(\mu)$. The corrected merge is
\[
\begin{aligned}
&M_0((G_A,H_A,B_A),(G_B,H_B,B_B))\\
&\quad=\left(G_A+G_B,
\frac{G_AH_A+G_BH_B}{G_A+G_B},
\frac{G_AB_A+G_BB_B}{G_A+G_B}\right).
\end{aligned}
\]
For three blocks $A,B,C$, the left bracketing gives first the state for $AB$ and then
\[
G_{(AB)C}=G_A+G_B+G_C,
\]
\[
H_{(AB)C}=\frac{(G_A+G_B)H_{AB}+G_CH_C}{G_A+G_B+G_C}
=\frac{G_AH_A+G_BH_B+G_CH_C}{G_A+G_B+G_C},
\]
\[
B_{(AB)C}=\frac{(G_A+G_B)B_{AB}+G_CB_C}{G_A+G_B+G_C}
=\frac{G_AB_A+G_BB_B+G_CB_C}{G_A+G_B+G_C}.
\]
The right bracketing gives the same three expressions after replacing the first merge by $BC$. Therefore, the corrected merge is associative on triples. Induction over the number of leaves proves that every ordered binary tree has root state
\[
\left(\sum_iG_i,
\frac{\sum_iG_i h_i}{\sum_iG_i},
\frac{\sum_iG_i B_i}{\sum_iG_i}\right).
\]

For the induction, the one-leaf state is $(G_i,h_i,B_i)$. Suppose the two child subtrees have disjoint leaf sets $I$ and $J$ and satisfy
\[
\begin{aligned}
(G_I,H_I,B_I)
&=\left(\sum_{i\in I}G_i,
\frac{\sum_{i\in I}G_ih_i}{\sum_{i\in I}G_i},
\frac{\sum_{i\in I}G_iB_i}{\sum_{i\in I}G_i}\right),\\
(G_J,H_J,B_J)
&=\left(\sum_{j\in J}G_j,
\frac{\sum_{j\in J}G_jh_j}{\sum_{j\in J}G_j},
\frac{\sum_{j\in J}G_jB_j}{\sum_{j\in J}G_j}\right).
\end{aligned}
\]
Substitution into the corrected merge gives
\[
\begin{aligned}
G_{I\cup J}&=\sum_{\ell\in I\cup J}G_\ell,\\
H_{I\cup J}
&=\frac{G_IH_I+G_JH_J}{G_I+G_J}
=\frac{\sum_{\ell\in I\cup J}G_\ell h_\ell}
       {\sum_{\ell\in I\cup J}G_\ell},\\
B_{I\cup J}
&=\frac{G_IB_I+G_JB_J}{G_I+G_J}
=\frac{\sum_{\ell\in I\cup J}G_\ell B_\ell}
       {\sum_{\ell\in I\cup J}G_\ell}.
\end{aligned}
\]
This proves the induction step and the root formula.

We use the following auxiliary notation. A second-order affine normal form is a binary update of the form
\[
B_{AB}=\frac{G_AB_A+G_BB_B}{G_A+G_B}+\beta(A,B),
\]
where $\beta(A,B)$ is a zero-mass second-order coefficient assigned to the merge of adjacent blocks $A$ and $B$. For three adjacent blocks, expanding the left bracketing gives
\[
\begin{aligned}
B^{(AB)C}
&=\frac{G_A+G_B}{G_A+G_B+G_C}\left(\frac{G_AB_A+G_BB_B}{G_A+G_B}+\beta(A,B)\right)\\
&\quad+\frac{G_CB_C}{G_A+G_B+G_C}+\beta(AB,C)\\
&=\frac{G_AB_A+G_BB_B+G_CB_C}{G_A+G_B+G_C}
+\frac{G_A+G_B}{G_A+G_B+G_C}\beta(A,B)+\beta(AB,C).
\end{aligned}
\]
Similarly,
\[
\begin{aligned}
B^{A(BC)}
&=\frac{G_AB_A}{G_A+G_B+G_C}
+\frac{G_B+G_C}{G_A+G_B+G_C}\left(\frac{G_BB_B+G_CB_C}{G_B+G_C}+\beta(B,C)\right)\\
&\quad+\beta(A,BC)\\
&=\frac{G_AB_A+G_BB_B+G_CB_C}{G_A+G_B+G_C}
+\frac{G_B+G_C}{G_A+G_B+G_C}\beta(B,C)+\beta(A,BC).
\end{aligned}
\]
Subtracting the two terms cancels the weighted-average term. We call the remaining coefficient the elementary coboundary of $\beta$ and 
\[
\begin{aligned}
(\delta\beta)(A,B,C)
&=\frac{G_A+G_B}{G_A+G_B+G_C}\beta(A,B)+\beta(AB,C)\\
&\quad -\frac{G_B+G_C}{G_A+G_B+G_C}\beta(B,C)-\beta(A,BC).
\end{aligned}
\]
Thus, the elementary rotation coefficient is $\delta\beta$.

\subsection{Proof of the normal-form correction}

The uncorrected square-root chart has second-order update
\[
B_{AB}=\frac{G_AB_A+G_BB_B}{G_A+G_B}+\beta_f(A,B),
\]
where
\[
\beta_f(A,B)=\tau(G_A^2,G_B^2;H_A/p,H_B/p)(H_B-H_A).
\]
The corrected chart subtracts the non-averaging second-order term $\beta_f(A,B)$ at each merge; its elementary coboundary $\delta\beta_f$ is the corresponding rotation coefficient. Hence, the corrected second-order update is exactly the weighted average
\[
B^0_{AB}=\frac{G_AB^0_A+G_BB^0_B}{G_A+G_B}.
\]
The calculation in the previous subsection proves associativity of the corrected merge. By induction on the tree, the root corrected state is
\[
G_T=\sum_iG_i,
\quad
H_T=\frac{\sum_iG_ih_i}{\sum_iG_i}
\,\,\,\,\text{and}\,\,\,\,\,
B_T^0=\frac{\sum_iG_iB_i}{\sum_iG_i}.
\]
The induction has one-leaf base case $(G_i,h_i,B_i)$ and the internal-node step follows by substituting the two child formulas into the corrected weighted average.

When $H_T/p,B_T^0/p\in L^\infty(\mu)$, the polynomial $p+\eps H_T+\eps^2B_T^0$ is a density for all sufficiently small $|\eps|$ and represents the corrected coefficient state through second order. This coefficient-level construction does not itself provide a finite-$\eps$ corrected fusion operator or a binary remainder. If such an operator is defined separately and, at every generated merge, has remainder $R_\eps(A,B)$ satisfying $\|R_\eps(A,B)\|\le C|\eps|^3$ uniformly over the relevant corrected local chart, finite induction over a fixed-size tree propagates that conditional bound to its root whenever weighted averaging is nonexpansive in the chosen norm.

\section{Proofs for Section~\ref{sec:bounds}}\label{app:bounds}
The admissible tree set $\mathcal U$ is finite, all generated merges are nondegenerate, and the local expansion of Theorem~\ref{thm:tree} holds uniformly over $\mathcal U$ after shrinking the local chart neighborhood if necessary.

\subsection{Proof of Proposition~\ref{prop:diameter-asymptotic}}

For every tree $T\in\mathcal U$, Theorem~\ref{thm:tree} gives
\[
\Fw{T}=p+\eps H+\eps^2B_T+r_T^\eps \,\,\,\,\text{and}\,\,\,\,\,
\sup_{T\in\mathcal U}\|r_T^\eps\|=o(\eps^2).
\]
The uniform remainder follows because $\mathcal U$ is finite and each tree has $n-1$ internal nodes. For two trees $T,T'\in\mathcal U$,
\[
\Fw{T}-\Fw{T'}=\eps^2(B_T-B_{T'})+(r_T^\eps-r_{T'}^\eps).
\]
The reverse triangle inequality yields
\[
\left|\|\Fw{T}-\Fw{T'}\|-\eps^2\|B_T-B_{T'}\|\right|
\le \|r_T^\eps-r_{T'}^\eps\|
\le 2\sup_{S\in\mathcal U}\|r_S^\eps\|=o(\eps^2).
\]
Taking the supremum over all pairs in $\mathcal U$ gives
\[
\left|\AODfw^{\mathcal U}(p_1^\eps,\ldots,p_n^\eps)-\eps^2\Delta_2^{\mathcal U}\right|
\le 2\sup_{S\in\mathcal U}\|r_S^\eps\|=o(\eps^2),
\]
which proves the density-level expansion.

For the quantity-of-interest form, set
\[
\|g\|_{\QoI}=\sup_{\varphi\in\QoI}\left|\int_\Omega\varphi g\,d\mu\right|.
\]

Let
\[
M_{\QoI}=\sup_{\varphi\in\QoI}\|\varphi\|_\infty<\infty.
\]
For every $g\in L^1(\mu)$, H\"older's inequality gives
\[
\|g\|_{\QoI}
\le M_{\QoI}\|g\|_1.
\]
Thus, convergence in $L^1(\mu)$ implies convergence in the seminorm $\|\cdot\|_{\QoI}$. The reverse triangle inequality for seminorms gives
\[
\left|\|g+h\|_{\QoI}-\|g\|_{\QoI}\right|\le \|h\|_{\QoI}.
\]
Apply the last inequality with $g=\eps^2(B_T-B_{T'})$ and $h=r_T^\eps-r_{T'}^\eps$. Uniformly over $T,T'\in\mathcal U$,
\[
\begin{aligned}
\|r_T^\eps-r_{T'}^\eps\|_{\QoI}
&\le M_{\QoI}\|r_T^\eps-r_{T'}^\eps\|_1\\
&\le 2M_{\QoI}\sup_{S\in\mathcal U}\|r_S^\eps\|_1
=o(\eps^2).
\end{aligned}
\]
Taking the supremum over the finite set of tree pairs gives
\[
\AODfwQ^{\mathcal U}(p_1^\eps,\ldots,p_n^\eps)=\eps^2\Delta_{2,\QoI}^{\mathcal U}+o(\eps^2).
\]

\subsection{Proof of Theorem~\ref{thm:rotation-bound}}

Let $T=T_0,T_1,\ldots,T_m=T'$ be a rotation path whose trees all belong to $\mathcal U$. At step $j$, let $\rho_j$ be the elementary rotation sending $T_{j-1}$ to $T_j$. The two involved trees agree outside one local triple of adjacent blocks. At the rotated subtree root, let
\[D_0\in\{\Curv_{\rho_j},-\Curv_{\rho_j}\}\quad\text{and}\quad
\|D_0\|=\|\Curv_{\rho_j}\|,\]
according to the orientation of the rotation.
Both bracketings have the same transformed mass and first-order coefficient. Consider one ancestor merge of the rotated subtree with an unchanged sibling $S$. If $D_\ell$ is the second-order coefficient difference entering that ancestor, the two non-averaging terms agree because they depend only on the common first-order coefficients and transformed masses. The parent difference is therefore
\[
D_{\ell+1}
=\lambda_\ell D_\ell,
\qquad
\lambda_\ell=\frac{G_V}{G_V+G_S},
\]
where $V$ is the child containing the rotated subtree. Iteration through all ancestors gives
\[
B_{T_j}-B_{T_{j-1}}
=\left(\prod_{\ell}\lambda_\ell\right)D_0.
\]
More generally, the norm of the surrounding coefficient context is bounded by the stated modulus $L_{\rho_j}$. Thus,
\[
\|B_{T_j}-B_{T_{j-1}}\|\le L_{\rho_j}\|\Curv_{\rho_j}\|.
\]
The telescoping identity
\[
B_{T'}-B_T=\sum_{j=1}^m(B_{T_j}-B_{T_{j-1}})
\]
and the triangle inequality give
\[
\|B_T-B_{T'}\|\le \sum_{j=1}^m\|B_{T_j}-B_{T_{j-1}}\|
\le \sum_{j=1}^mL_{\rho_j}\|\Curv_{\rho_j}\|.
\]
The first inequality is therefore valid for every admissible path. Taking the infimum over all such paths between fixed endpoints and then the supremum over $T,T'\in\mathcal U$ gives the stated diameter bound.

\subsection{Proof of Corollary~\ref{cor:global-bound}}

For every admissible path of length at most $m_n$, the hypotheses give
\[
\sum_{\rho\in\gamma}L_\rho\|\Curv_\rho\|
\le \sum_{\rho\in\gamma}Lc
\le Lm_nc.
\]
Theorem~\ref{thm:rotation-bound} gives $\Delta_2^{\mathcal U}\le Lm_nc$. Proposition~\ref{prop:diameter-asymptotic} gives
\[
\AODfw^{\mathcal U}(p_1^\eps,\ldots,p_n^\eps)=\eps^2\Delta_2^{\mathcal U}+o(\eps^2)
\le \eps^2Lm_nc+o(\eps^2).
\]

\subsection{Proof of Theorem~\ref{thm:lower}}

Let $T$ and $T'$ be the endpoint trees of the elementary rotation edge $\rho$. Since both trees belong to $\mathcal U$, the restricted aggregation-order variation satisfies
\[
\AODfw^{\mathcal U}(p_1^\eps,\ldots,p_n^\eps)
\ge d(\Fw{T},\Fw{T'}).
\]
By definition of $K_\rho$ and by the tree expansion,
\[
\Fw{T'}-\Fw{T}=\eps^2K_\rho+o(\eps^2).
\]
The assumed metric expansion gives
\[
d(\Fw{T},\Fw{T'})=\eps^2\|K_\rho\|+o(\eps^2).
\]
Combining the two terms, dividing by $\eps^2$, and taking $\liminf_{\eps\downarrow0}$ gives
\[
\liminf_{\eps\downarrow0}\eps^{-2}\AODfw^{\mathcal U}(p_1^\eps,\ldots,p_n^\eps)
\ge \|K_\rho\|.
\]
When the rotation occurs at the root, Proposition~\ref{prop:rotation} identifies $K_\rho$ with $-\Curv(A,B,C)$ for the orientation $(AB)C\to A(BC)$ and with $\Curv(A,B,C)$ for the reverse orientation.

\subsection{Proof of Proposition~\ref{prop:scalar}}

Let the common supplied weight be $w_0$, and let $G_0=g(w_0)>0$ be the common transformed leaf mass. The lower merge of two leaves carries transformed mass $2G_0$. At the upper merge, the two-leaf subtree and the remaining leaf carry transformed masses $2G_0$ and $G_0$, respectively. Their endpoint balancing weights are therefore $(2G_0)^2=4G_0^2$ and $G_0^2$, respectively, and thus have ratio $4:1$.

For every $\lambda>0$,
\[
\frac{\sqrt{\lambda b}}{\sqrt{\lambda a}+\sqrt{\lambda b}}
=\frac{\sqrt b}{\sqrt a+\sqrt b},
\]
so $t_0(\lambda a,\lambda b)=t_0(a,b)$. In addition,
\[
\begin{aligned}
\tau(\lambda a,\lambda b;u_A,u_B)
&=\frac{A_3L_3\{\lambda a t_0^3+\lambda b(1-t_0)^3\}}
{2A_2J\{\lambda a t_0+\lambda b(1-t_0)\}}\\
&=\tau(a,b;u_A,u_B).
\end{aligned}
\]
Taking $\lambda=G_0^2$ shows that the lower and upper coefficients equal the coefficients computed for transformed leaf mass one. The first-order averages are also unchanged because
\[
\frac{G_0H_A+G_0H_B}{2G_0}=\frac{H_A+H_B}{2}.
\]
The complete witness derivation therefore gives
\[
\begin{aligned}
B^{(12)3}
=k\alpha^2\left(-\frac34,-\frac34,\frac32\right) \quad\text{and}\quad
B^{1(23)}
=k\alpha^2\left(\frac32,-\frac34,-\frac34\right).
\end{aligned}
\]
Subtracting yields

\[
B^{(12)3}-B^{1(23)}
=-\frac{3\alpha^2}{4}\frac{f'''(1)}{f''(1)}(1,0,-1).
\]
Since $\alpha\ne0$ and $f'''(1)\ne0$, the stated vector is nonzero. Scalar calibration therefore leaves a nonzero second-order tree-dependent coefficient on the finite three-state input family.

\end{document}